\documentclass[]{interact}

\usepackage{epstopdf}
\usepackage[utf8x]{inputenc}

\usepackage[numbers,sort&compress]{natbib}
\bibpunct[, ]{[}{]}{,}{n}{,}{,}
\renewcommand\bibfont{\fontsize{10}{12}\selectfont}

\theoremstyle{plain}
\newtheorem{theorem}{Theorem}[section]

\newtheorem{proposition}[theorem]{Proposition}

\theoremstyle{definition}

\theoremstyle{remark}
\newtheorem{remark}{Remark}

\usepackage[cmintegrals,cmbraces]{newtxmath}
\usepackage{setspace}
\usepackage{amsfonts}
\usepackage{amssymb}
\usepackage{booktabs}
\usepackage{mathrsfs}
\usepackage{caption}
\DeclareCaptionLabelSeparator{none}{ } 
\usepackage{calc}
\usepackage[hidelinks]{hyperref}
\usepackage{multicol}
\usepackage{pdfpages}
\usepackage{color}
\usepackage{enumerate}
\usepackage{listings}
\usepackage{lscape}
\usepackage{siunitx}
\usepackage{empheq}
\usepackage{mathtools,cases}
\usepackage{mdframed}
\usepackage{tikz}
\usepackage{framed}
\usepackage{amsfonts}
\usepackage{environ}
\usepackage{tabularx}
\usepackage{multirow}
\usepackage{enumitem}

\usepackage{subcaption}

\DeclareRobustCommand{\rchi}{{\mathpalette\irchi\relax}}
\newcommand{\irchi}[2]{\raisebox{\depth}{$#1\chi$}} 

\definecolor{myred}{RGB}{179,0,0}

\newcommand*{\modif}[1]{#1}

\definecolor{mygreen}{RGB}{28,172,0} 
\definecolor{mylilas}{RGB}{170,55,241}

\newcommand{\molarmass}[0]{\mathscr{M}} 
\newcommand{\tauxreaction}[1]{\tau_{#1}} 
\newcommand{\tauxproduction}[1]{\omega_{#1}} 
\newcommand{\species}[1]{\mathcal{E}_{#1}} 
\newcommand{\hstandardmol}[1]{h^0_{#1,mol}}
\newcommand{\hstandardmass}[1]{h^0_{#1}}
\newcommand{\massflux}[0]{\ensuremath m} 

\newcommand{\partialdershort}[2]{ \partial_{#2}{#1} }
\newcommand{\derivshort}[2]{ d_{#2}{#1} }

\newcommand{\partialder}[2]{ \dfrac{\partial {#1}}{\partial {#2}} }

\newcommand{\rhocomplet}[0]{\rho}
\newcommand{\Tcomplet}[0]{T}
\newcommand{\ucomplet}[0]{u}
\newcommand{\xcomplet}[0]{x}
\newcommand{\Ycomplet}[0]{Y}

\newcommand{\rhowave}[0]{\ensuremath\hat{\rho}} 
\newcommand{\Twave}[0]{\ensuremath\hat{T}}
\newcommand{\Ywave}[0]{\ensuremath\hat{Y}}
\newcommand{\uwave}[0]{\ensuremath\hat{u}}
\newcommand{\xwave}[0]{\ensuremath\hat{x}}

\newcommand{\Qgas}[0]{\ensuremath Q} 
\newcommand{\Qgasmol}[0]{\ensuremath Q_{mol}} 
\newcommand{\Qpyro}[0]{\ensuremath Q_{p}} 

\newcommand{\Gi}[0]{{G^{1}}} 
\newcommand{\Gis}[0]{{G^{1}_{(s)}}} 
\newcommand{\Gig}[0]{{G^{1}_{(g)}}} 

\newcommand{\Gii}[0]{{G^{2}}} 
\newcommand{\Giig}[0]{{G^{2}_{(g)}}} 

\newcommand{\Ps}[0]{{P_{(s)}}} 

\newcommand{\nomS}[0]{target interface balance}
\newcommand{\nomDG}[0]{effective interface balance}
\newcommand{\nomXi}[0]{interface balance mismatch}

\newcommand{\referentielRG}[0]{$\mathscr{R}_G$}

\newcommand{\cmin}{c_{min}}
\newcommand{\cmax}{c_{max}}
\newcommand{\Tslim}{T_{s,min}}

\newcommand{\thetamin}{\theta_{s,min}}

\newtheorem{assumption}{H}

\NewEnviron{curlyeqset}[2]{
  \begin{empheq}[left={\empheqlbrace}]{align}
  \setlength{\jot}{#2}
  \BODY
  \ignorespaces
  \end{empheq}
}
\NewEnviron{curlyeqset*}[2]{
  \begin{empheq}[left={\empheqlbrace}]{align*}
  \setlength{\jot}{#2}
  \BODY
  \ignorespaces
  \end{empheq}
}
\NewEnviron{curlyeqsetat}[2]{
  \begin{empheq}[left={\empheqlbrace}]{alignat=#1}
  \setlength{\jot}{#2}
  \BODY
  \ignorespaces
  \end{empheq}
}
\NewEnviron{eqsetat}[2]{
  \begin{alignat}{#1}
  \setlength{\jot}{#2}
  \BODY
  \ignorespaces
  \end{alignat}
}
\NewEnviron{curlyeqsetat*}[2]{
  \begin{empheq}[left={\empheqlbrace}]{alignat*=#1}
  \setlength{\jot}{#2}
  \BODY
  \ignorespaces
  \end{empheq}
}

\NewEnviron{hypothese}[0]{
  \textbf{\BODY}
}
\title{Semi-analytical stationary solution for solid propellants}
\author{Laurent Francois}

\newmdenv[
  topline=false,
  bottomline=false,
  skipabove=\topsep,
  skipbelow=\topsep
]{siderules}

\newenvironment{packed_list}{
 \setlist{nolistsep}
 \begin{itemize}[noitemsep]}
 {\end{itemize}}

\begin{document}


\title{Travelling wave mathematical analysis and efficient numerical resolution for a one-dimensional model of solid propellant combustion}

\author{
\name{Laurent François\textsuperscript{1,2}, Joël Dupays\textsuperscript{1}, Dmitry Davidenko\textsuperscript{1}, Marc Massot\textsuperscript{2}}
\affil{\textsuperscript{1}ONERA, DMPE, 6 Chemin de la Vauve aux Granges, 91120 Palaiseau, France\\
       \textsuperscript{2}CMAP, CNRS, Ecole polytechnique, Institut Polytechnique de Paris, Route de Saclay, 91128 Palaiseau Cedex, France
}
}

\maketitle

\begin{abstract}
\noindent
We investigate a model of solid propellant combustion involving surface pyrolysis coupled to finite activation energy gas phase combustion.
Existence and uniqueness of a travelling wave solution are established by extending dynamical system tools classically used for premixed flames,
dealing with the additional difficulty arising from the surface regression and pyrolysis.
An efficient shooting method allows to solve the problem in phase space without resorting to space discretisation nor fixed-point Newton iterations.
The results are compared to solutions from a CFD code developed at ONERA, assessing the efficiency and potential of the method, and the impact of the modelling assumptions
is \modif{evaluated} through parametric studies.

\end{abstract}


\section{Introduction}

Solid propellant combustion is a key element in rocket propulsion and has been extensively studied since the 1950s including at ONERA
\cite{deluca1992,novozhilov1992, lengelle, barrere}. This particular problem involves
a solid phase and a gas phase, separated by an interface (surface of the solid). The solid is heated up by thermal conduction and radiation from the gas phase.
At its surface, the solid propellant is decomposed through a pyrolysis process, and the resulting pyrolysis products are gasified and injected in the gas phase.
All these phenomena will be gathered under the name ``pyrolysis'' for simplicity.
The interface regresses and the injected species react and form a flame which heats back the solid, allowing for a sustained combustion. 
It is essential to understand the physics of this phenomenon to allow for clever combustion chamber designs and efficient solid rocket motors.
A key element is the regression speed of the propellant surface, and its dependence on the combustion chamber conditions.

Many models have been developed, with essentially two levels of description. On the one side, there exists 
analytical models, which directly give a formula for the steady regression speed and allow a qualitative description and global understanding
of the physics at the cost of some restrictive assumptions \cite{williamsDBW, BDPmodel, WSBmodel}.
On the other side, one can find detailed models, that require high-fidelity numerical
resolution with spatial and temporal discretisations (CFD), giving a very detailed representation of the physics, both for quasi-steady and unsteady evolutions. 
The main analytical models in steady regime are the Denison-Baum-Williams (DBW) model \cite{williamsDBW}, the Beckstead-Derr-Price (BDP) model \cite{BDPmodel}
and the Ward-Son-Brewster (WSB) model \cite{WSBmodel}.
They mainly give the regression speed as a function of surface temperature, pressure and initial temperature of the propellant
using a one-dimensional approach.
They usually assume the pyrolysis is concentrated in a narrow region close to the surface and the gas phase only contains two species: one reactant resulting from the pyrolysis, and one product.
There is only one global reaction which transforms the reactant into the product.
The DBW and BDP models assume that the activation energy $E_a$ of the gas phase reaction is very high. This allows the splitting of the gas phase into two separate zones:
the convection-diffusion zone and the reaction-diffusion zone, starting at the flame stand-off distance $x_f$ (model-specific).
The equations can be solved in each zone separately and linked at the interface between the two, yielding an \modif{explicit or implicit} expression for the burning mass flow rate $\massflux$.
On the opposite, the WSB model assumes that $E_a$ is zero, which often leads to better agreement with experimental results \cite{Brewster_simplified_combustion}. 
Assuming a unitary Lewis number, several equations can be derived,
which require simple fixed-point iterations to determine the regression speed.
All these models give 
relations between the propellant physical characteristics and the physical state of the propellant and gas flow
(surface temperature, regression speed). 
They allow for a global understanding of the phenomenon. In all these models, the equations describing the physics of both
phases can only be solved for a unique value of the regression speed $c$, also called the eigenvalue\footnote{\modif{The name eigenvalue is adopted here for two reasons. First, it is historically used in the papers on solid propellant theory as well as in the laminar flame theory, mainly without quotes. Second, we will investigate a non-linear eigenvalue problem for an  elliptic operator with a non-linear source term as well as a non-linear dependency of the regression rate on the surface temperature, and thus a specific case of a general non-linear eigenvalue problem on the whole real line.
It bears some similarity with the eigenvalue problem of a second order linear elliptic operator such as the Laplace operator on a compact interval with proper boundary conditions \cite{schatzman2002}: we look for both an eigenfunction of space related to a real eigenvalue of the operator. Its extension to the non-linear case has also been studied in the literature \cite{henderson1997}.}} of the problem.

At the other end of the spectrum, numerical methods solving a comprehensive set of equations, e.g. \cite{massa_2D,meynet2005simulation, meynet2006}, use much
less restrictive assumptions, but they are computationally expensive and 
may encounter convergence difficulties. Parametric studies are therefore costly.

It would \modif{thus} be interesting to design an easy-to-use and yet precise analytical or semi-analytical model,
which would not require as many assumptions as the existing analytical models, in particular for the gas phase reaction activation energy,
thus remaining closer to the physics, 
amenable to a full theoretical study of existence and uniqueness, setting the mathematical basis of the model, and which can be resolved efficiently
using a specific numerical method.
Such a model already exists for travelling combustion waves in laminar premixed flames and has been studied for quite some time, for example in \cite{livreZeldovich}.
It is based on a phase-plane representation of a simplified combustion problem with unitary Lewis number, two species and a single reaction;
the existence and uniqueness of a travelling wave profile can be proved relying on dynamical system theory.
The combustion wave speed is shown to be a key parameter for which only one value allows the simplified problem to be solved. 
This value can be determined numerically
through a shooting method, for any value of the activation energy of the gas phase reaction. 

In this paper, we investigate a specific model of solid propellant combustion. It 
involves pyrolysis in an infinitely thin zone at the interface coupled to solid regression and homogeneous gas phase combustion described
by one global reaction with finite activation energy, heat and molecular diffusion at unitary Lewis number. The first contribution compared to existing analytical models
is the relaxation of the \modif{assumption on} 
activation energy, which is not any more considered to be either zero or very large.
The model is derived from a detailed system of equations that describes 
the evolution of the temperature of the solid propellant, the evolution of the gas phase, and the pyrolysis of the propellant.  
The model takes into account thermal expansion and density changes in the gas phase.
We study the existence and uniqueness of a travelling wave solution of this system, that is we look for a temperature profile and a wave velocity $c$,
the so-called eigenvalue or regression \modif{velocity}.

In the case of a solid propellant, Verri \cite{Verri} presented a demonstration of the uniqueness of the travelling wave solution and its stability using a
different approach, without modelling the gas phase, but only considering the gas heat feedback as a function of the regression speed with specific mathematical properties.
This means the gas heat feedback was assumed to have a unique value for a given regression speed,
which tacitly means that the gas phase temperature profile was also assumed unique, although this was not investigated in his paper.
His framework was also more restrictive, as only exothermic surface reactions were considered, or weakly endothermic ones. No numerical method was
developed to determine the solution profile and regression speed.
Johnson and Nachbar \cite{1963JohnsonNachbarUnicity} have also analysed the mathematical behaviour and uniqueness of the eigenvalue for the burning of a monopropellant,
but they assumed that the surface temperature is a given constant. They showed that for any reasonable value of this temperature, a single regression speed exists such that
the complete problem is solved.
In the present paper however, we aim at proving the existence and uniqueness of the solution for a variable surface temperature determined from the regression speed via a 
pyrolysis law, with proper representation of the gas phase, and a non-trivial  coupling condition at the interface. Consequently additional difficulties appear, 
which are overcome through a detailed dynamical system study of the associated heteroclinic orbit in phase space.
Our approach is an extension of the one used by Zeldovich \cite{livreZeldovich} in the study of laminar flames, with the addition of a variable interface temperature.
Interpretation of the behaviour of the system in phase space brings a better understanding of the role of the interface and the influence of the different parameters.

The phase space approach also naturally leads to the development of a very efficient 
numerical shooting method to iteratively find the speed of the wave and ultimately its profile
with arbitrary precision.
\modif{Other approaches for example based on the topological degree theory have been used to establish the existence and uniqueness of the travelling wave solution for laminar flames \cite{degreTopologique_Berestycki,giovangigli1999} and may be applicable to the solid propellant case with fewer assumptions, however they do not allow for the derivation of a numerical method to generate solution profiles.
We therefore follow the phase space approach and develop a shooting method to compute the burning rate and solution profiles}. We propose to assess its efficiency and potential by first verifying the proposed numerical strategy in comparison with a CFD code developed at ONERA with the same level of modelling and then follow up with a parametric study of the influence of the activation energy of the chemical reaction in the gaseous phase. Eventually, the improvement of our approach compared to some \modif{of the previously} mentioned analytical models is investigated as well as the influence of some assumptions such as the unitary Lewis number. To the best of our knowledge, no such study has yet been presented.

The paper is organized as follows: In section 2, we first introduce the generic equations describing the gas flow, the solid phase and the coupling conditions. 
Adding a series of assumptions, we gradually simplify
the system, while still retaining the most important physical features. We introduce a travelling wave solution and derive the equations that describe
the wave profile. The impact of the wave \modif{velocity} $c$ is explained. We derive some general relations to obtain characteristic values for a non-dimensionalisation of the 
problem. In section 3, extending  Zeldovich's approach for laminar flames, we prove the existence and uniqueness of the self-similar temperature profile and
wave speed by a dynamical system approach and focus on how to handle the specific  solid-gas interface flux condition. A physical interpretation of the result is provided.
Section 4 is devoted to the presentation of the multiple algorithms used in the numerical resolution based on a shooting method,
and the comparison for various levels of modelling assumptions with a CFD code developed at ONERA.
A conclusion and an assessment of the efficiency and potential of the approach is given in section 5.

\section{General modelling, proper set of simplifications and travelling wave formalism}

In this section we start by presenting the general assumptions usually made for advanced models in high-fidelity simulations, as well as the associated set of equations.
We introduce some additional assumptions also used in analytical models such as in \cite{williamsDBW} but, as opposed to these models, no assumption is made
about the activation energy of the gas phase.
We derive a set of equations that is simple enough to envision an analytical study of travelling wave solutions.
Although this system is much simpler than the original system, it is expected that the relaxation of the gas phase activation energy may allow for a more
realistic picture of the combustion of a solid propellant, as compared to existing analytical models.

\subsection{Derivation of the model, related assumptions and travelling wave formulation}
\label{part:derivation_modele_base}
Composition and temperature variations in the solid phase decomposition zone and in the gas flame structure are often important in the direction normal to the
burning propellant surface, so that it is common to adopt a one-dimensional approach which greatly simplifies the mathematical developments.
The phenomenon is studied in the Galilean reference frame \referentielRG~ and a schematic representation is provided in Figure \ref{fig:modele1D}.
\begin{figure}
  \centering
  \includegraphics[width=0.6\linewidth]{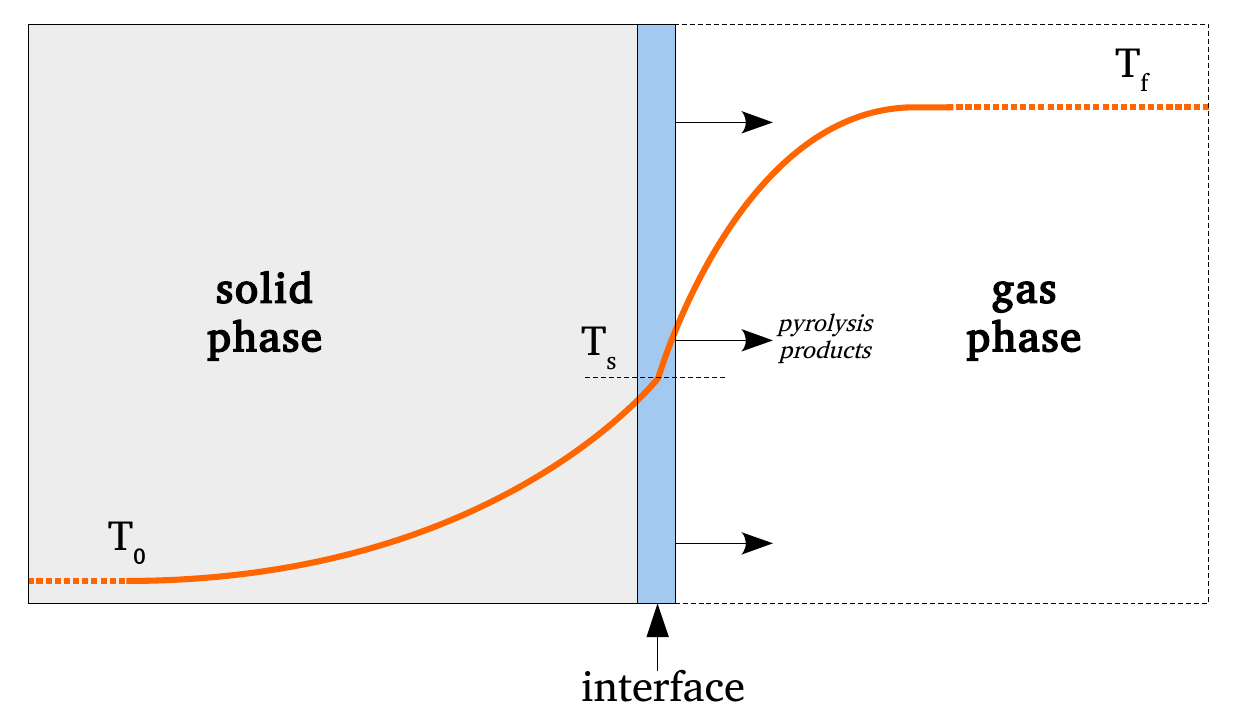}
  \caption{One-dimensional model of solid propellant combustion}
  \label{fig:modele1D}
\end{figure}

The solid and gas phases are separated by a superficial degradation zone, which is a transition zone where both gas and liquid species are observed.
This zone is usually thin, typically one micron or less \cite{TanakaBecksteadAP} for ammonium perchlorate, a few dozen microns for
HMX/RDX \cite{DavidsonBecksteadHMX, DavidsonBecksteadRDX},
and its thickness decreases as pressure increases.
It has therefore been common to represent this zone as an infinitely thin interface, as in all previously mentioned analytical models.
We adopt the same representation, with \modif{$\sigma(t)$} the position of the interface at time $t$. The behaviour of this zone is given by a pyrolysis law.
The solid phase (the propellant) is semi-infinite and is located between $-\infty$ and
$x=\sigma(t)$.
The gas phase is also semi-infinite and is located between $x=\sigma(t)$ and $+\infty$.
The instantaneous regression speed is $c(t) = \derivshort{\sigma}{t}(t)$.
Most of the solid propellant numerical models assume the following:
\begin{assumption}[]
The solid phase is inert, incompressible and inelastic.
No species diffusion takes place in the solid.
Far from the burning surface, the solid phase is at its initial temperature, $T(-\infty) = T_0$.
All gradients vanish at $x=-\infty$.
\end{assumption}

\begin{assumption}[]
The gas phase is  constituted of a mixture of reacting ideal gases in the low-Mach number limit and the pressure $P$ does not vary with time.
\end{assumption}

\begin{assumption}[]
\label{assum:interface}
No species or heat accumulation takes place at the interface.
The temperature is continuous across the interface and its value is denoted $T_s(t)$.
The gasification process is controlled by a pyrolysis reaction concentrated at the interface.
The mass flow rate of gaseous species expelled by the solid phase through the pyrolysis reaction is given by a pyrolysis law of the form:
\begin{equation}
  \label{eq:pyrolysisLawOrigin}
  \massflux = f(T_s,T_0,P)
\end{equation}
The mass flow rate is $\mathcal{C}^\infty$ with respect to $T_s$ and satisfies the property $\partialdershort{f}{T_s} > 0$,
i.e. the mass flow rate increases with increasing surface temperature.
\end{assumption}

The pyrolysis law used in our numerical applications is $\massflux = A_p \exp \left(-T_{ap}/T_s\right)$.
This simple law can be extended to include a dependence on pressure (typically $P^n$) and surface temperature ($T_s^\beta$ with $\beta>0$)
in the pre-exponential factor $A_p$.
This law is frequently used for stationary as well as transient studies of solid propellant combustion,
although it ignores some potentially important effects which only appears in more comprehensive pyrolysis relations, deduced for
instance from activation energy asymptotics with zeroth-order reaction inside the solid \cite{Brewster_simplified_combustion}.
All the conclusions made in this paper remain valid for any other pyrolysis law $\massflux$, as long as the mass flow rate is \modif{increasing} 
with $T_s$. \modif{In our study $T_0$ and $P$ are constants, therefore the mass flow rate will be considered a function of $T_s$ only for clarity.}

An additional assumption is proposed on the pyrolysis law for mathematical requirements, without impacting the physics. It is similar to the cold boundary difficulty
resolution \cite{coldBoundaryKarman1953,coldBoundaryMarble1956,coldBoundary1991} and will allow for an easier theoretical analysis.

\begin{assumption}[]
\label{assum:modif_pyro}
 The propellant will not be consumed at $T_s \leq T_0$, $T_0$ being in all realistic cases close to the ambient temperature.
 That behaviour is usually not depicted exactly by the pyrolysis laws found in the literature, which usually only tends to $0$ as $T_s$ tends to $0$.
 Therefore we introduce a slightly modified pyrolysis law that contains a cut-off so that $\massflux$ smoothly goes to 0 as $T_s$ approaches $T_0$:
  \begin{equation}
  \label{eq:modifiedPyrolysisLaw}
   \massflux = f(T_s, T_0, P) ~\phi(T_s-T_0) 
 \end{equation}
 with $\phi$ a \modif{smooth function} 
 such that $\phi(y) = 0$ for $y \leq 0$ and $\phi(y)$ quickly reaches a value of $1$ as $y$ becomes greater than 0. The function 
 $\phi$ can typically be a sigmoid function. 
 \end{assumption}

The \modif{smooth cut-off} introduced here is solely used in our theoretical analysis, to \modif{facilitate} the proof of existence and uniqueness of the \modif{travelling wave solution}.
The numerical shooting method presented in Section \ref{section:numCFD} does not require \modif{such a cut-off}.

To further simplify the problem, we assume:
\begin{assumption}
 Radiative effects are neglected.
\end{assumption}

This means that no external flux, e.g. laser flux, \modif{no radiation from the gas phase and no radiative heat loss from the solid} are considered. Radiative heat losses were shown to
allow for two different \modif{travelling wave} solutions to be found at a given pressure \cite{deflagrationLimitJohnson}, with only one being stable. 
\modif{The inclusion of such phenomena in the modelling and its impact on the results obtained in the present investigation are discussed in the conclusion.}

From the mathematical point of view, we assume a certain regularity of the solution profiles as stated below.
\begin{assumption}
\label{assum:regulariteSolution}
 All solution components are \modif{smooth functions of $x$} 
 in each separate phases, but may be discontinuous at the interface.
\end{assumption}

Using the heat equation to model the evolution of the temperature inside the solid, and the Navier-Stokes equations with reactions and species transport
for the gas phase allows us to describe the evolution of our system.
The solid phase is represented by its temperature $T(x,t)$ and its constant density $\rho_s$.
The gas phase is described by the density $\rho(x,t)$, the constant pressure $P$, the flow speed $u(x,t)$, and the temperature $T(x,t)$.
The reactive aspect of the flow with $n_e$ species (symbol $\species{i}$, $i \in \llbracket 1,n_e \rrbracket$) is taken into account with the addition of the transport equations
for the species mass fractions $Y_i(x,t)$ and
the addition of the volumetric heat release as a source term in the energy equation. The mass production rate of the i-th species per unit volume is 
$\tauxproduction{i}$, in $\textnormal{kg.m}^{-3}\textnormal{.s}^{-1}$.
We consider $n_r$ chemical reactions of the form:
$ \sum\nolimits_{i=1}^{n_e} \nu_{i,r}' \species{i} \rightarrow \sum\nolimits_{i=1}^{n_e} \nu_{i,r}'' \species{i}$.
We introduce $\nu_{i,r} = \nu_{i,r}'' - \nu_{i,r}'$, the global stoichiometric coefficient.
The reaction rate of the r-th reaction is $\tauxreaction{r}$, in $\textnormal{mol.m}^{-3}\textnormal{.s}^{-1}$.
It is typically a generalized Arrhenius law dependent on species concentrations, temperature and pressure. 
We have the relation $\tauxproduction{i} = \molarmass_i \sum\nolimits_{r=1}^{n_r} \nu_{i,r} \tauxreaction{r}$.
The molar enthalpy of the i-th species is $h_{i,mol}(T) = \hstandardmol{i} + \molarmass_i c_{p,i}(T-T_{std})$ in $\textnormal{J.mol}^{-1}$ 
with $\molarmass_i$ the molar mass of this species, $c_{p,i}$ its specific heat, and $T_{std}$ the standard temperature
at which all standard molar enthalpies $\hstandardmol{i}$ are defined.
The enthalpies and standard enthalpies per unit mass are written $h_i$ and $\hstandardmass{i}$ respectively.
The gas and solid phases are coupled at the interface through boundary conditions obtained by integration of the energy and transport equations around the interface.
Assuming the pyrolysis process is concentrated at the interface, we introduce the ``injection'' mass fractions \modif{$Y_{i}(\sigma^-)$}
for the different gaseous species, which indicates the mass fractions obtained after pyrolysis directly at the interface, before entering the gas phase.
Our original system of equations is given here in the Galilean reference frame \referentielRG.

\paragraph*{The full model\newline}

\label{system:base:equations}
The solid phase at \modif{$x<\sigma(t)$} is subject to:
\begin{equation}
    \rhocomplet_s c_s \partialdershort{\Tcomplet}{t} - \partialdershort{\left(\lambda_s \partialdershort{\Tcomplet}{\xcomplet}\right)}{\xcomplet} = 0 \label{eq:base:solid}
\end{equation}
\noindent with the boundary condition for the resting temperature of the solid:
\begin{equation}
    \Tcomplet({-\infty}) = \Tcomplet_0 \label{eq:base:BCs:minf}
\end{equation}

\noindent The gas phase at \modif{$x>\sigma(t)$} is subject to the following partial differential equations:
\begin{curlyeqset}{1}{1pt}
    & \partialdershort{\rhocomplet}{t} + \partialdershort{\rhocomplet \ucomplet}{\xcomplet} = 0  \label{eq:base:continuity}\\
    & \partialdershort{\rhocomplet \Ycomplet_i}{t} + \partialdershort{\left(\rhocomplet (\ucomplet + V_i) \Ycomplet_i\right)}{\xcomplet} =
	      \tauxproduction{i}
	      \label{eq:base:species}\\
    & \rhocomplet c_p \partialdershort{ \Tcomplet}{t} + \rhocomplet c_p \ucomplet \partialdershort{T}{x}
    - \partialdershort{\left(\lambda_g\partialdershort{\Tcomplet}{\xcomplet}\right)}{\xcomplet}
    + \rho \partialdershort{\Tcomplet}{\xcomplet} \sum\limits_{i=1}^{n_e} c_{p,i} Y_i V_{i}
    = -\sum\limits_{i=1}^{n_e} h_i \tauxproduction{i}  \label{eq:base:energy}
\end{curlyeqset}

\noindent with the following conditions at $x=+\infty$, ensuring the gas phase reactions are complete:
\begin{curlyeqset}{1}{1pt}
  &\partialdershort{\Tcomplet}{\xcomplet}(+\infty) = 0\label{eq:base:BCs:gradpinfT}\\
  &\partialdershort{\Ycomplet}{\xcomplet}(+\infty) = 0\label{eq:base:BCs:gradpinfY}
\end{curlyeqset}

\noindent Both phases are coupled at the interface by the following conditions:
\begin{curlyeqset}{1}{1pt}
  &\Tcomplet_{\sigma^-} = \Tcomplet_{\sigma^+} = \Tcomplet_s \label{eq:base:BCs:interfaceContinuityT}\\
  &\left(\lambda_s \partialdershort{\Tcomplet}{x}\right)_{\sigma^-} = \massflux \Qpyro + \left( \lambda_g \partialdershort{\Tcomplet}{x} \right)_{\sigma^+} \label{eq:base:BCs:interfaceBilanT}\\
  &\left(\massflux \Ycomplet_i\right)_{\sigma^-} =  \left( \massflux \Ycomplet_i + J_i \right)_{\sigma^+} ~~~\forall~i ~\in~ \llbracket 1, n_e \rrbracket\label{eq:base:BCs:interfaceBilanY}
\end{curlyeqset}

with the ideal gas law:
\begin{equation}
     \rhocomplet = \dfrac{P}{R \Tcomplet \sum\limits_{i=1}^{n_e} \frac{\Ycomplet_i}{\molarmass_i}}
\end{equation}

Finally, the pyrolysis mass flow rate $\massflux$ is given by the pyrolysis law \eqref{eq:modifiedPyrolysisLaw}, \modif{and $\sigma$ is governed by $\derivshort{\sigma}{t} = c(t) = -\massflux(\Tcomplet_s(t)) / \rhocomplet_s$, with initial condition $\sigma(0)=\sigma_0$.}\\

This system involves the following variables: $R$ the ideal gas constant, $c_p$ the mixture-averaged specific heat, $V_{i}$ is the diffusion velocity of the i-th species,
$\Qpyro$ the heat of the pyrolysis reaction per unit mass of solid propellant consumed.
Equation \eqref{eq:base:BCs:interfaceBilanT}
means that the heat conducted from the gas phase into the solid and the heat generated by the pyrolysis process (if the pyrolysis is exothermic) are used to heat up the solid propellant and sustain the combustion.
Equation \eqref{eq:base:BCs:interfaceBilanY} is the species balance, i.e. the flow rate of the i-th species generated by the pyrolysis is equal to the flow rate of
this species leaving the surface in the gas phase, minus the species diffusion flow rate $J_i = \rho V_i Y_i$.

We introduce the following set of additional assumptions, also shared by the classical analytical models:
\begin{assumption}
  The specific heat of the solid $c_s$ is constant.
  The solid phase thermal conductivity $\lambda_s$ is constant.
\end{assumption}

\begin{assumption}
\label{assum:gasphase}
  The gas phase contains two species: the reactant $\Gi$ and the product $\Gii$, with mass fractions $Y_1$ and $Y_2$.

  There is only one irreversible reaction $\nu_1' \Gig \rightarrow \nu_2'' \Giig$, whose reaction rate $\tauxreaction{}$ is positive or 0.

  Both species have the same molar mass $\molarmass$ and therefore opposite global stoichiometric coefficients.
  
  The species specific heats $c_{p,i}$ are all equal and constant: $c_{p,i}=c_p ~ \forall i$, with $c_p$ the constant gas specific heat.
  
  No binary species diffusion takes place in the gas phase, and the molecular diffusion is represented by Fick's law:
  $J_i = \rho V_i Y_i = -\rho D_i \partialdershort{Y_i}{x}$.
    
  The species diffusion coefficients $D_i$ are equal, $D_i=D_g ~ \forall i$. 
  
  The species $\Gi$ is completely consumed at $x=+\infty$. All gradients are zero at $x=+\infty$.
 \end{assumption}
  
\begin{assumption}
\label{assum:injection}
 The pyrolysis reaction transforms the solid propellant $P$ into the species $\Gi$.
\end{assumption}
Using these assumptions, we can simplify the equations.
We introduce $D_{th} = {\lambda_g}/(\rho c_p)$ the thermal diffusivity of the gas, $D_s= \lambda_s/(\rho_s c_s)$ the thermal diffusivity of the solid propellant,
and the stoichiometric coefficient $\nu = \nu_1 = -\nu_2 = - \nu_1'<0$.
Having only two species, we replace the transport equation for $Y_2$ by the global mass conservation $Y_2 = 1 - Y_1$.
We introduce $\Qgasmol = \molarmass \left( \nu_1' \hstandardmass{1} - \nu_2'' \hstandardmass{2} \right) = \molarmass \nu (\hstandardmass{2}-\hstandardmass{1})$,
the molar heat of the reaction in the gas phase.


As we aim at studying steady, self-similar combustion waves, we look for a solution in the form of a travelling wave $f(x,t) = \hat f(x-ct)$ for all the variables, \modif{which moves at a time-independent velocity $c$ in the reference frame \referentielRG .}
This is equivalent to performing the variable change $ \xwave = x - ct$, as described in \cite{theseShihab}. \modif{The regression velocity of the propellant surface is $c$}, it should thus be considered negative in our study.
If we perform this variable change in our equations, the travelling wave we are looking for becomes a stationary solution. In particular,
the interface is at the constant abscissa \modif{$\xwave=\sigma_0$, which we assume to be zero for clarity}.
To highlight the fact that the variables associated with these new equations are different from the previous ones, we use the notation ``$~\hat{\cdot}~ $'',
and we introduce $\Ywave=\Ywave_1$ for the sake of simplicity. Overall, we obtain the following system of equations.

\paragraph*{The simplified travelling wave model\newline}
\label{system:onde:equations}
For $\xwave<0$:
\begin{equation}
    -c \derivshort{\Twave}{\xwave} - D_s \derivshort{\Twave}{\xwave\xwave} = 0 \label{eq:onde:solid}
\end{equation}
    For $\xwave>0$:
    \begin{curlyeqset}{1}{1pt}
        & - c \derivshort{\rhowave}{\xwave} + \derivshort{(\rhowave \uwave)}{\xwave} = 0  \label{eq:onde:continuity}\\
        & \rhowave (\uwave-c) \derivshort{\Ywave}{\xwave} - \derivshort{}{\xwave}\left( \rhowave D_g \derivshort{\Ywave}{\xwave} \right)
	    = \nu \molarmass \hat{\tauxreaction{}} \label{eq:onde:species}\\
        & \rhowave (\uwave-c) \derivshort{\Twave}{\xwave} - \derivshort{}{\xwave}\left( \rhowave D_{th} \derivshort{\Twave}{\xwave} \right)  = 
            \dfrac{\hat{\tauxreaction{}} \Qgasmol}{c_p}\label{eq:onde:energy}
    \end{curlyeqset}
    with the ideal gas law:
    \begin{equation}
     \rhowave = \dfrac{P \molarmass}{R \Twave}
    \end{equation}

\noindent The boundary conditions are:
    \begin{curlyeqset}{1}{1pt}
	&\Twave({-\infty}) = \Twave_0 \label{eq:onde:BCs:minf}\\
        &\Twave(0^-) = \Twave(0^+) = T_s \label{eq:onde:BCs:Tinterface}\\
        &\left(\lambda_s \derivshort{\Twave}{\xwave}\right)_{0^-} = \massflux \Qpyro + \left( \lambda_g \derivshort{\Twave}{\xwave} \right)_{0^+} \label{eq:onde:BCs:bilanInterface}\\
        &\left(\massflux \Ywave \right)_{0^-} =  \left( \massflux \Ywave -\rhowave D_g \derivshort{ \Ywave}{x} \right)_{0^+}\label{eq:onde:BCs:bilanYInterface}\\
        &\derivshort{\Twave}{\xwave}(+\infty) = 0\label{eq:onde:BCs:gradpinfT}\\
        &\derivshort{\Ywave}{\xwave}(+\infty) = 0\label{eq:onde:BCs:gradpinfY}
    \end{curlyeqset}

\noindent As before, the pyrolysis mass flow rate $\massflux$ is given by the pyrolysis law \eqref{eq:modifiedPyrolysisLaw}.

\paragraph*{\textbf{ \textit{ Key steps to obtain the set of simplified equations} } }

 In the full system, 
 the first two terms of all gas phase transport equations can be simplified by using the continuity equation, for example:
 $$\partialdershort{\rho Y_i}{t} + \partialdershort{\rho u Y_i}{x} = \rho \partialdershort{Y_i}{t} + \rho u \partialdershort{Y_i}{x}$$
 
 In the gas phase energy equation \eqref{eq:base:energy}, we can expand the term of heat diffusion caused by chemical diffusion,
 using our assumptions that all species specific heats are equal:
 $\rho\partialdershort{T}{x} \sum\limits_{i=1}^{n_e} c_{p,i} Y_i V_i = \rho c_p \partialdershort{T}{x} \sum\limits_{i=1}^{n_e} Y_i V_i = 0$,
 by definition of the diffusion velocities.
 
 The term of heat production due to the single chemical reaction can be simplified as follows:
 $$-\sum_{i=1}^{n_e} h_i \tauxproduction{i} = -\sum_{i=1}^{n_e} \tauxproduction{i} \left(\hstandardmass{i} + c_p (T-T_{std})\right)
 = -\tauxreaction{} \sum_{i=1}^{n_e} \nu_i \molarmass_i \left(\hstandardmass{i} + c_p (T-T_{std})\right)$$
 
 Using our assumption that all species molar masses are equal, the term in $c_p$ disappears according to mass conservation ($\sum\limits_{i=1}^{n_e} \nu_i \molarmass_i = 0$):
 $$ -\tauxreaction{} \sum_{i=1}^{n_e} \nu_i \molarmass_i \left(\hstandardmass{i} + c_p (T-T_{std})\right)
 = -\tauxreaction{} \molarmass \sum_{i=1}^{n_e} \nu_i \hstandardmass{i}
 = \tauxreaction{} \Qgasmol$$
 
 The last step is to perform a variable change, such that the interface remains at a constant abscissa.
 Therefore we introduce the new space variable $\xwave = \xwave(\xcomplet, t) = \xcomplet - \int_0^t c(\eta)d\eta$,
 where $c(t)$ is the instantaneous regression \modif{velocity} 
 of the interface.
 In the Galilean reference frame \referentielRG, the interface lies at the abscissa 
 $\sigma(t) = \sigma(0) + \int_0^t c(\eta)d\eta$, therefore following our variable change, the new interface abscissa is
 $\hat{\sigma}(t) = \sigma_0$, and we assume, without any loss of generality, that this position is $0$.
 
 If, for a function $f(\xcomplet,t)$, 
 we introduce $\hat{f}(\xwave, t)$ such that $\hat{f}(\xwave(\xcomplet,t), t) = f(\xcomplet,t)$, we can derive the following relations:
 $$(\partialdershort{f}{\xcomplet})_t = (\partialdershort{\hat{f}}{\xwave})_t
 \qquad\qquad
 (\partialdershort{f}{t})_x = -c(t) (\partialdershort{\hat{f}}{\xwave})_t + (\partialdershort{\hat{f}}{t})_{\xwave}
$$
\noindent Using these relations and our steady-state assumption $(\partialdershort{\hat{f}}{t})_{\xwave} = 0$,
we can transform all our partial differential equations and obtain the simplified system of equations.
The change of space variable results in the introduction of an additional
convection term at uniform speed.

\begin{remark}
Note that in the unsteady case, the variable change is not equivalent to a change of reference frame. Indeed the gas flow velocity is still the one observed in the
Galilean reference frame \referentielRG. If $c$ varies with time, no inertial body force appears.
\end{remark}

\begin{remark}
The regression \modif{velocity}    
$c$ appears in these equations in various manners:
in the convective terms,
in the interface boundary conditions,
in the mass flow rate at the interface $\massflux = -\rho_s c$,
and through $T_s$, as the mass flow rate is linked to surface temperature through \eqref{eq:modifiedPyrolysisLaw}.
\end{remark}

For further simplifications, we introduce additional assumptions, which do not alter our simplified system of equations:
\begin{assumption}
\label{assum:cpcs}
The specific heats of the solid and gas phases are equal: $c_p=c_s$.
\end{assumption}

\begin{assumption}
\label{assum:lewis}
The Lewis number $\mathrm{Le} = D_{th}/D_g$ is 1 in the gas phase, i.e. the heat and species diffusions are equivalent. We introduce
$D = D_g = D_{th}$. 
\end{assumption}

\begin{remark}
Even if questionable, the Assumption H\ref{assum:cpcs} is often used in the literature
\cite{massa_2D, jackson_2D, Brewster_simplified_combustion, Miller_idealized} and
the results obtained are still quantitatively correct.
The main effect of this assumption is that $\Qpyro$ is a constant which only depends on the standard enthalpies.
\end{remark}
\begin{remark}
\label{remark:Qpyro}
The complete pyrolysis reaction $\Ps \rightarrow \Gig$ can be decomposed into two successive reactions:
\begin{itemize}
 \item $\Ps \rightarrow \Gis$, the transformation
    of the solid propellant $\Ps$ into the pyrolysis product at solid state $\Gis$, with the heat of reaction
    $Q_s = \hstandardmass{\Ps} - \hstandardmass{\Gis} + (c_s - c_{P_{\Gis}}) (T_s-T_{std})$,
    with $c_{P_{\Gis}}$ the specific heat of the pyrolysis product $\Gi$ at solid state.
 \item $\Gis \rightarrow \Gig$, the sublimation of the solid pyrolysis product $\Gis$ into $\Gig$, at constant temperature $T_s$, with the latent heat
    $L_v = h_{\Gig}(T_s) - h_{\Gis}(T_s) = (\hstandardmass{\Gig} - \hstandardmass{\Gis}) + (c_p - c_{P_{\Gis}})(T_s - T_{std})$
\end{itemize}
This leads to $ \Qpyro = Q_s - L_v =  \hstandardmass{\Ps} - \hstandardmass{\Gig} + (c_s-c_p)(T_s-T_{std})$ which depends linearly on $T_s$.
On the contrary, the heat of reaction $\Qgas$ for  $\Gig \rightarrow \Giig$ in the gas phase does not depend
on  temperature as both species have the same specific heat.

The assumption $c_s=c_p$ makes the upcoming theoretical analysis much easier by removing the dependence of $\Qpyro$ on $T_s$.
However, the numerical method presented further in this paper does not rely on this assumption.
\end{remark}

\subsection{Conservation properties}
Starting from a detailed modelling of the different processes at stake, we have introduced gradual simplifications based on several physical assumptions.
We may now perform simple mathematical manipulations on the simplified travelling wave model to derive several balance equations.
These considerations will allow us to obtain characteristic scales from which
dimensionless variables can be formed. In order to avoid a notation overload, we drop the symbol ``$~\hat{\cdot}~$''.

\begin{proposition}
\label{prop:bilan:massflux}
The conservation of the mass flow rate implies for $x > 0$: 
$$\rho(x) (u(x) -c) = -\rho_s c = \massflux$$
\end{proposition}
\begin{proof}
We integrate the continuity equation \eqref{eq:onde:continuity} in the gas phase from $0^+$ to $x$. We obtain:
$$\rho(x) ( u(x) - c) = \rho(0^+) ( u(0^+) - c)$$
Following Assumption H\ref{assum:interface}, no accumulation takes places at the interface.
Hence the gas mass flow rate must be equal to the rate of propellant mass consumption $-\rho_s c$, which is the proposed result.
\end{proof}

\begin{remark}
 We may now explain why the momentum equation is not considered.
 In the low-Mach framework, this equation would, in steady-state, essentially reduce to: $\massflux \derivshort{u}{x} = \derivshort{\tilde{P}}{x}$,
 with $\tilde{P}$ the hydrodynamic pressure, which is a pressure perturbation of the order of $\mathrm{Ma}^2$, with $\mathrm{Ma}$ the Mach number.
 Due to the one-dimensionality of our approach, the velocity field is directly related to the spatial evolution of $\rho$ through the continuity equation.
 Hence, the momentum equation is not needed to determine $u$. We may still use it to determine the hydrodynamic pressure field.
 In particular, we would find that the hydrodynamic pressure is increasing with $x$, and that the overall pressure variation across the gas phase
 is $\Delta P = -\massflux \Delta u$. Typically we obtain $-10$ Pa, which is considerably lower than the average pressure (around $1$ to $10$ MPa).
 This legitimates our assumption of uniform $P$.
 Had we used a multi-dimensional approach, we would not have been able to decouple the velocity field from the hydrodynamic pressure field, and we would have needed to include
 the momentum equation.
\end{remark}

\begin{proposition}
\label{prop:bilan:omega}
The complete consumption of $\Gi$ implies:
$$ \int_0^{+\infty}  \tauxreaction{}(T(x), Y(x)) dx = - \dfrac{\massflux}{\molarmass \nu}$$
\end{proposition}

\begin{proof}
We integrate the species transport equation \eqref{eq:onde:species} in the gas phase from $0^+$ to $+\infty$, utilizing the
mass flow rate balance result $\massflux = \rho(u-c)$.
As all gradients are zero at $+\infty$, we get:
$$ -\massflux Y(0^+) + \rho(0^+) D \derivshort{Y}{x}(0^+) = \molarmass \nu \int_{0^+}^{+\infty} \tauxreaction{}(T(x), Y(x)) dx$$
Using equation \eqref{eq:onde:BCs:bilanYInterface}, the left-hand side is equal to $-\massflux Y(0^-)$. Following H\ref{assum:injection}, we have $Y(0^-)=1$,
therefore we obtain the proposed result.
\end{proof}
%
We introduce $\Qgas= -\Qgasmol / (\nu \molarmass)$ the heat of reaction in  the gas phase per unit mass of $\Gig$ consumed.
As $\nu<0$, $\Qgas$ and $\Qgasmol$ are both positive.
\begin{proposition}
\label{prop:bilan:T}
The burnt gas temperature at $x=+\infty$ is $T_f = T_0 +  (\Qgas + \Qpyro)/c_p$.
\end{proposition}
\begin{proof}
 Integrating the energy equation \eqref{eq:onde:energy} in the gas phase between $0$ and $+\infty$ and using Proposition \ref{prop:bilan:omega}, we obtain:
 $$\massflux \left( T_f - T_s \right) = - \dfrac{\lambda_g}{c_p} \derivshort{T}{x}(0^+) + \dfrac{\Qgasmol}{c_p} \int_0^{+\infty}  \tauxreaction{}(T(x), Y(x)) dx
 = - \dfrac{\lambda_g}{c_p} \derivshort{T}{x}(0^+) + \dfrac{\massflux \Qgas}{c_p}$$
 
 Integrating the heat equation in the solid phase \eqref{eq:onde:solid} between $-\infty$ and $0$ and using Proposition \ref{prop:bilan:massflux}, we can write:
 $$\massflux (T_s-T_0) = +\dfrac{\lambda_s}{c_s} \derivshort{T}{x}(0^-)$$
 
 Using the interface boundary condition \eqref{eq:onde:BCs:Tinterface} and Assumption H\ref{assum:cpcs}, we can combine both energy balances and obtain the proposed result.
\end{proof}

\begin{remark}
 \label{remark:energyBalanceCsCp}
 If we do not assume $c_s = c_p$, the balance reads:
 $$ T_f = \left(1-\dfrac{c_s}{c_p}\right)T_s + \dfrac{c_s}{c_p} T_0 + \dfrac{\Qgas + \Qpyro}{c_p}$$
 This formula appears in many papers, however the dependence in $T_s$ is fictitious and may be misleading \cite{IBIRICU1975185}.
 Indeed it is expected that the complete energy balance does not depend on the mass flow rate. The dependence is removed when using the standard enthalpies to express
 $\Qpyro$ as a function of $T_s$ (see remark \ref{remark:Qpyro}):
 $$T_f = \dfrac{c_s}{c_p}T_0 + \dfrac{\hstandardmass{\Ps} - \hstandardmass{\Giig}}{c_p} + \left(1-\dfrac{c_s}{c_p}\right) T_{std}$$
 This expression can also be directly obtained by performing a simple energy balance between $-\infty$ and $+\infty$, neglecting the kinetic energy.
\end{remark}
\begin{proposition}
\label{prop:bilan:enthalpy}
 Under Assumption H\ref{assum:lewis}, we can define a combustion enthalpy which is constant in the gas phase:
 $$h = -\dfrac{\Qgasmol}{\nu\molarmass} Y + c_p (T-T_0) = c_p (T_f-T_0)$$
\end{proposition}

\begin{proof}
\newcommand{\Th}[0]{\ensuremath\check{T}} 
\newcommand{\Yh}[0]{\ensuremath\check{Y}} 
We introduce $\Yh = { Y \Qgasmol}/({\nu \molarmass})$ and
$\Th = {c_p (T-T_{0})}$.
We convert the equations \eqref{eq:onde:species} and \eqref{eq:onde:energy} to our new variables.
Using Assumption H\ref{assum:lewis}, introducing $\beta = \lambda_g/c_p = \rho D$ which is a constant and $h=\Th-\Yh$, we obtain:

$$ \massflux \derivshort{h}{x} = \beta \derivshort{h}{xx} $$
We can integrate this expression between the interface $0^+$ and $x$:
$$ h(x) = h(0^+) + \derivshort{h}{x}(0^+) \int_0^{x} \exp\left(\dfrac{\massflux y}{\beta}\right) dy$$
Using the boundary conditions at $+\infty$, we get $\derivshort{h}{x}(+\infty)=0$. Alternatively, $Y$ and $T$ being bounded, $h$ is bounded too.
Both considerations lead to $\derivshort{h}{x}(0^+) = 0$, hence we get:
$h(x) = h(+\infty) = c_p(T_f-T_0)$.
\end{proof}

\subsection{Dimensionless equations}

Using the equations from the simplified travelling wave model 
and the results obtained in the previous subsection, we can now write dimensionless equations for our problem.
We introduce $L$ a constant length scale.

\paragraph*{Dimensionless system}
\label{system:adim:equations}
 Introducing
 $\tilde{x} = \dfrac{x}{L}$,
 $\tilde{c}=\dfrac{cL}{D_s}$,
 $\tilde{\theta} = \dfrac{T(\tilde{x})-T_0}{T_f-T_0}$,
 $\eta = \dfrac{\lambda_s}{\lambda_g}$,
 using the notation $\cdot'=\derivshort{\cdot}{\tilde{x}}$ and $\tilde{\gamma}(\tilde{x}) = \tilde{\theta}'(\tilde{x})$ we have:
 
 \begin{curlyeqset}{1}{1pt}
 & \tilde{\theta}' = \tilde{\gamma} \label{eq:adim:relation_theta_gamma}\\
  & \tilde{c}\tilde{\gamma} + \tilde{\gamma}' = 0 ~&\textnormal{for $\tilde{x}<0$} \label{eq:adim:solid}\\
  & \eta\tilde{c} \tilde{\gamma} + \tilde{\gamma}' = -\tilde{\Psi}  ~~~&\textnormal{for $\tilde{x}>0$} \label{eq:adim:gas}
 \end{curlyeqset}

 with the dimensionless heat source term:
  \begin{equation}
   \label{eq:definition:psi}
   \tilde{\Psi}(\tilde{x}) = \dfrac{L^2 \Qgasmol}{\lambda_g (T_f-T_0)} \tilde{\tauxreaction{}}(\tilde{\theta}(\tilde{x})) ~\geq~ 0
  \end{equation}

The associated boundary conditions are:\newline
%
  \begin{minipage}{0.4\textwidth}
    \begin{curlyeqset}{1}{1pt}
    &\tilde{\theta}(-\infty) = 0 \label{eq:adim:BCs:minf}\\
    &\tilde{\theta}(0^-) = \tilde{\theta}(0^+) = \tilde{\theta}_s(\tilde{c}) \label{eq:adim:BCs:Tinterface}\\
    &\tilde{\gamma}(0^+) - \eta \tilde{\gamma}(0^-)  = \tilde{S}(\tilde{c})  \label{eq:adim:BCs:bilanInterface}\\
    &\tilde{\theta}(+\infty) = 1  \label{eq:adim:BCs:pinf}
    \end{curlyeqset}
  \end{minipage}
  \begin{minipage}{0.5\textwidth}
    \begin{curlyeqset}{1}{1pt}   
    &\tilde{\gamma}(-\infty) = 0 \label{eq:adim:BCs:gradminf}\\
    &\tilde{\gamma}(+\infty) = 0 \label{eq:adim:BCs:gradpinf}
    \end{curlyeqset}
  \end{minipage}

\vspace{2em}

  with the \nomS{}:
  \begin{equation}
   \label{eq:definition:S}
   \tilde{S}(\tilde{c}) = \eta\dfrac{\Qpyro}{\Qpyro + \Qgas} \tilde{c}
  \end{equation}


\paragraph*{\textbf{ \textit{ Key steps to obtain the dimensionless system} } }
\noindent
For the solid phase, we take equation \eqref{eq:onde:solid}, divide it by $(T_f - T_0)$ to let $\theta$ appear,
and switch the spatial derivatives from $x$ to $\tilde{x}$ ($L \derivshort{}{x} = \derivshort{}{\tilde{x}}$);
we then multiply it by $L^2/D_s$ and use the definition of $\gamma$ to obtain \eqref{eq:adim:solid}. For the gas phase, Proposition \ref{prop:bilan:enthalpy}
allows us to \modif{express $Y$ as a function of $T$}. 
Therefore the reaction rate $\tauxreaction{}(T,Y)$ can be expressed as a function of temperature only
$\tilde{\tauxreaction{}}(\tilde{\theta})$.
Equation \eqref{eq:adim:gas} is then obtained from \eqref{eq:onde:energy} in a similar fashion as for the solid phase.
Equation \eqref{eq:adim:BCs:bilanInterface} can be obtained from equation \eqref{eq:onde:BCs:bilanInterface} after the same kind of process, with
$\tilde{S}(\tilde{c}) = c {L \rho_s \Qpyro(T_s)}/({\lambda_g (T_f - T_0)}) = \tilde{c} {\rho_s D_s \Qpyro(T_s)}/({\lambda_g(T_f - T_0)})$.
Recalling the definitions of $\eta$, $D_s$, we obtain:
\begin{equation}
 \label{eq:definition:Scpcs}
 \tilde{S}(\tilde{c}) = \dfrac{\eta \Qpyro(T_s)}{c_s(T_f - T_0)}\tilde{c}
\end{equation}

Using H\ref{assum:cpcs} and the global energy balance from Proposition \ref{prop:bilan:T}, we get \eqref{eq:adim:BCs:bilanInterface}
and \eqref{eq:definition:S}. All the other boundary conditions are directly \modif{obtained from the ones of the} simplified travelling wave model. 

\begin{remark}
\label{remark:impactC}
For a given value of $\tilde{c}$, $\tilde{\theta}_s$ is given by \eqref{eq:modifiedPyrolysisLaw}.
Therefore the first-order ODEs \eqref{eq:adim:relation_theta_gamma} and \eqref{eq:adim:solid} can be integrated from $\tilde{x}=-\infty$ to $0$, using the boundary conditions
\eqref{eq:adim:BCs:minf} and \eqref{eq:adim:BCs:Tinterface}, and the solution profiles for $\tilde{\theta}$ and $\tilde{\gamma}$ are unique.
Similarly the first-order ODEs \eqref{eq:adim:relation_theta_gamma} and \eqref{eq:adim:gas} may also be integrated from $\tilde{x}=+\infty$ to $0$, using the boundary conditions
\eqref{eq:adim:BCs:pinf} and \eqref{eq:adim:BCs:Tinterface}, and the solution profiles are also unique.
Boundary conditions \eqref{eq:adim:BCs:gradminf} and \eqref{eq:adim:BCs:gradpinf} are only introduced to emphasise the behaviour of the system at infinity, however they are not
mathematically required.
The difficulty arises from the interface thermal balance \eqref{eq:adim:BCs:bilanInterface} which overconstrains our system.
For a random value of $\tilde{c}$, it is likely that this condition will not be satisfied. However the dependence of this condition on $\tilde{c}$ through $\gamma$ and
the \nomS{} $\tilde{S}$ allows us to envision that some specific values of $\tilde{c}$ might lead to this condition being verified (hence the name ``target'' for $\tilde{S}$).
Therefore, the \modif{dimensionless regression velocity}    
  $\tilde{c}$ is a key variable and can be considered as an ``eigenvalue'' of the problem.
\end{remark}

\begin{remark}
\label{remark:psi}
  The dimensionless heat source term $\tilde{\Psi}$ has the same behaviour as the reaction rate $\tilde{\tauxreaction{}}$.
  It is positive for $\tilde{\theta} \in [0,1]$ and vanishes for $\tilde{\theta}=1$, since all the fuel is burned, i.e. $\tilde{\Psi}(1) = 0$.
\end{remark}

\begin{remark}
 The sign of the temperature gradient jump across the interface $\left[ \derivshort{T}{x}\right]_{0^-}^{0^+}$, or equivalently
 $\left[ \derivshort{\tilde{\theta}}{\tilde{x}}\right]_{0^-}^{0^+}$,
 depends on three factors:
 \begin{packed_list}
  \item $\eta = \lambda_s/\lambda_g$ the ratio of the thermal conductivities in the gas and in the solid
  \item $\Qpyro$ the reaction heat of the pyrolysis reaction, detailed in Remark \ref{remark:Qpyro}
  \item $T_s$, which is directly related to the regression rate
 \end{packed_list}
 
\noindent Let us underline that the presence of the ratio $\eta$ of thermal conductivities may have a strong impact on the sign of the jump.
 As an example, in a configuration where $\Qpyro = 0$, we have $\tilde{S}(\tilde{c})=0$.
 If $\eta>1$, then $\derivshort{T}{x}(0^+) > \derivshort{T}{x}(0^-)$, but if $\eta<1$, then $\derivshort{T}{x}(0^+) < \derivshort{T}{x}(0^-)$.
\end{remark}

\begin{remark}
\modif{With the reactant mass fraction $Y$ now removed from our set of variables, the species interface condition \eqref{eq:onde:BCs:bilanYInterface} is not considered any more. It can actually be shown that, as long as the thermal interface condition \eqref{eq:adim:BCs:bilanInterface} is satisfied, this condition is automatically fulfilled, even if $\mathrm{Le} \neq 1$.}
\end{remark}

\section{Existence and uniqueness of a travelling wave solution profile and velocity}

In this section, we will use the previously established dimensionless system to prove that there exists at least one value of the regression \modif{velocity}    
 $c$ such
that all boundary conditions can be satisfied and the complete travelling wave problem can be solved.
We then proceed to show that there is only one such value of $c$.
For the sake of simplicity, we drop the ``~$\widetilde{\cdot}$~'' notation.

\subsection{Monotonicity of the temperature profile}
A first step in our proof of existence is to show that the temperature profile is \modif{increasing} in the gas phase.
\begin{proposition}
\label{prop:monotonicity}
 \modif{The temperature $T$ is an increasing function of $x$ on $\mathbb{R}$.}
 \modif{There exists a bijection $g: \mathbb{R} \rightarrow [0,1]$ such that $g(x) = \theta$.}
\end{proposition}
\begin{proof}

This proposition is established by considering the behaviour of the temperature in the two phases successively.

{\bf Solid phase}
Let \modif{$x_0 \in (-\infty, 0)$} 
the position of a local extremum for $\theta$: $\gamma(x_0)=0$. From equation \eqref{eq:adim:solid} we get  $\gamma'(x_0)=0$.
If we integrate equation \eqref{eq:adim:relation_theta_gamma} from $-\infty$ to $x_0$, we get $\theta(x_0)=0$. Therefore no local extremum can be lower than $0$.
If a local maximum exists at $x_0$, $\theta(x_0)=0$. As $\theta(-\infty)=0$, there would then exist a local minimum $x_1<x_0$,
and we must have $\theta(x_1)<0$, which contradicts our previous finding.
Therefore no local extremum exists for $\theta$ in the solid phase.
As $c<0$ implies $\theta(0) = \theta_s>0$ (see equation \eqref{eq:modifiedPyrolysisLaw}), we can conclude that $\theta$ is \modif{increasing} 
in this phase.

{\bf Gas phase}
We want to prove that the temperature profile is monotonous and increasing in the gas phase.
Using a \textit{reductio ad absurdum}, \modif{let us} suppose that $\exists ~ x_0 ~/~ \gamma(x_0)=0$, local extremum or inflection point for $\theta$.
The energy equation \eqref{eq:adim:gas} then reads: $\theta''(x_0) = \gamma'(x_0) = -\Psi(\theta(x_0)) ~ < ~ 0$, which means that $x_0$ can only be a local maximum.
Consequently there exists a local minimum at \modif{$x_1 \in (x_0, +\infty)$} 
such that $\gamma(x_1)=0$ and $\gamma(x)<0 ~ \forall~ x \in (x_0, x_1)$.
We obviously have $\theta(x_1)<\theta(x_0)$. Integrating equation \eqref{eq:adim:gas} from $x_0$ to $x_1$ yields:
 $\eta c [\theta(x_1) - \theta(x_0)] = -\int_{x_0}^{x_1} \Psi(\theta(x)) dx$.
The left-hand side is strictly positive, but the right-hand side is strictly negative, consequently there exists no local maximum $x_0$.
Overall, $\theta$ does not have any local extremum in the gas phase, and as $\theta(+\infty) > \theta(0^+)$, we can conclude that
$\theta$ is monotonous and increasing in the gas phase.
This proof is the consequence of a much more general principle in the study of second order elliptic equations called the maximum principle \cite{berestycki1994}.

{\bf Overall Monotonicity}
The boundary condition $\theta(0^+) = \theta(0^-) = \theta_s$ and the requirement $\theta_s>0$ allow us to prove that
$\theta$ is increasing and strictly monotonous across both phases.
\modif{Therefore, we can build a bijection $g: \mathbb{R} \rightarrow [0,1]$ such that $g(x) = \theta$.} 
This proof is valid even if the regression \modif{velocity}    
 $c$ is such that the interface thermal balance \eqref{eq:adim:BCs:bilanInterface} is not satisfied.
\end{proof}
\vspace{.5cm}

We now make use of the monotonicity of $\theta$ to switch from a spatial point of view to a phase space one. The  bijection between $\theta$ and $x$ allows for
a variable change from $x$ to $\theta$ in our equations. Therefore, $\gamma$ may be considered a function of $\theta$. We also have the relation
$\derivshort{\gamma}{x} = \derivshort{\gamma}{\theta} \derivshort{\theta}{x} = \gamma\derivshort{\gamma}{\theta}$.
We can transform the dimensionless system into the following one, which we will use to
determine the orbit of our system in the phase plane $(\theta, \gamma)$.

\paragraph*{Reduced dynamical system for orbit evaluation}
  The dimensionless system 
  is equivalent to the following set of first-order ODEs and boundary conditions:
  
    \begin{curlyeqset}{1}{1pt}
     &c\gamma(\theta) ~~+ \gamma(\theta)\derivshort{\gamma}{\theta}(\theta) = 0 ~& \forall~ \theta \in [0, \theta_s(c)] \label{eq:portrait:solid}\\ 
     &\eta c \gamma(\theta) + \gamma(\theta)\derivshort{\gamma}{\theta}(\theta) = - \Psi(\theta) &\forall~ \theta \in [\theta_s(c), 1] \label{eq:portrait:gas}
    \end{curlyeqset}
   \begin{curlyeqset}{1}{1pt}
      &\gamma(0)=0 \label{eq:portrait:BC_minf}\\
      &\gamma(1)=0 \label{eq:portrait:BC_pinf}\\
      &\gamma(\theta_s^+) - \eta \gamma(\theta_s^-) = S(c) \label{eq:portrait:BC_interface}
   \end{curlyeqset}

\begin{remark}
This set of equations is similar to the one obtained by Zeldovich et al. \cite{livreZeldovich} for a homogeneous gaseous laminar flame. 
In this reference, the phase portrait of the temperature profile is also split in two parts.
The first one represents the part of the profile where the temperature is lower than an artificial cut-off temperature $\theta_{ignition}$,
below which the reaction rate $\Psi$ is forced to zero. This allows the ``cold boundary'' problem
\cite{coldBoundaryKarman1953,coldBoundaryMarble1956,coldBoundary1991} to be overcome.
The zone where $\theta<\theta_ {ignition}$ is purely a convection-diffusion zone.
The second part of the laminar flame phase portrait is the same as ours: the gas phase undergoes a reaction which creates a steep increase
in temperature before reaching the adiabatic combustion temperature behind the combustion wave. This is a convection-diffusion-reaction zone.
The two zones are joined using the continuity of the temperature profile and its gradient, as no reaction or heat accumulation takes place at the interface.
In our case, the first part of the phase portrait is not associated with a cut-off of the gas phase reaction rate, but with the fact that the solid phase is inert,
therefore it only heats up through thermal diffusion. Our problem thus differs in two ways from the laminar flame one. First, the pyrolysis process is concentrated
at the interface and causes a discontinuity of the temperature gradient, which depends on the wave \modif{velocity} $c$.
Secondly, the position $\theta_s$ of the interface in the phase portrait also varies with $c$, whereas $\theta_{ignition}$  is an arbitrary constant.
We can artificially make our problem equivalent to the laminar flame's one by forcing $\Qpyro=0$, $\eta=1$, $\theta_s = \theta_{ignition}$ (no pyrolysis law), $D_s=D_g$
(and $c_s=c_p$ as assumed in H\ref{assum:cpcs}).
\end{remark}

The rest of the study will be based on the analysis of the phase portrait of the system, i.e. the plot of $\gamma$ versus $\theta$.
Such a phase portrait is represented in Figure \ref{fig:schemPhasePortrait}.

\begin{figure}
\begin{minipage}[t]{0.425\textwidth}
\includegraphics[width=1.2\textwidth, trim={3cm 2cm 2cm 2cm},clip]{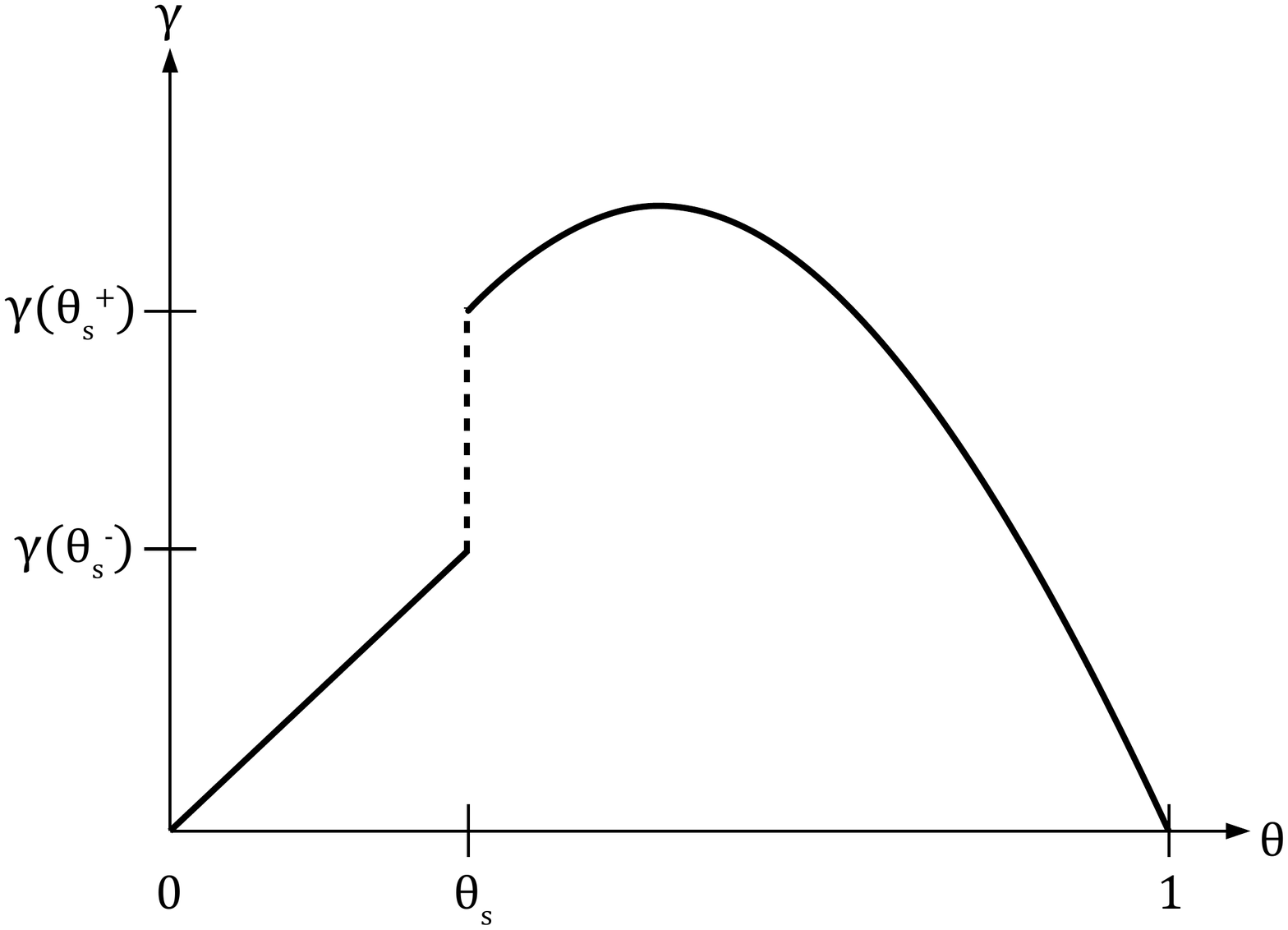}
\captionof{figure}{\label{fig:schemPhasePortrait} Schematic phase portrait in both phases}
\end{minipage}
\begin{minipage}[t]{0.05\textwidth}
~
\end{minipage}
\begin{minipage}[t]{0.425\textwidth}
\includegraphics[width=1.2\textwidth, trim={0cm 1cm 0cm 0cm},clip]{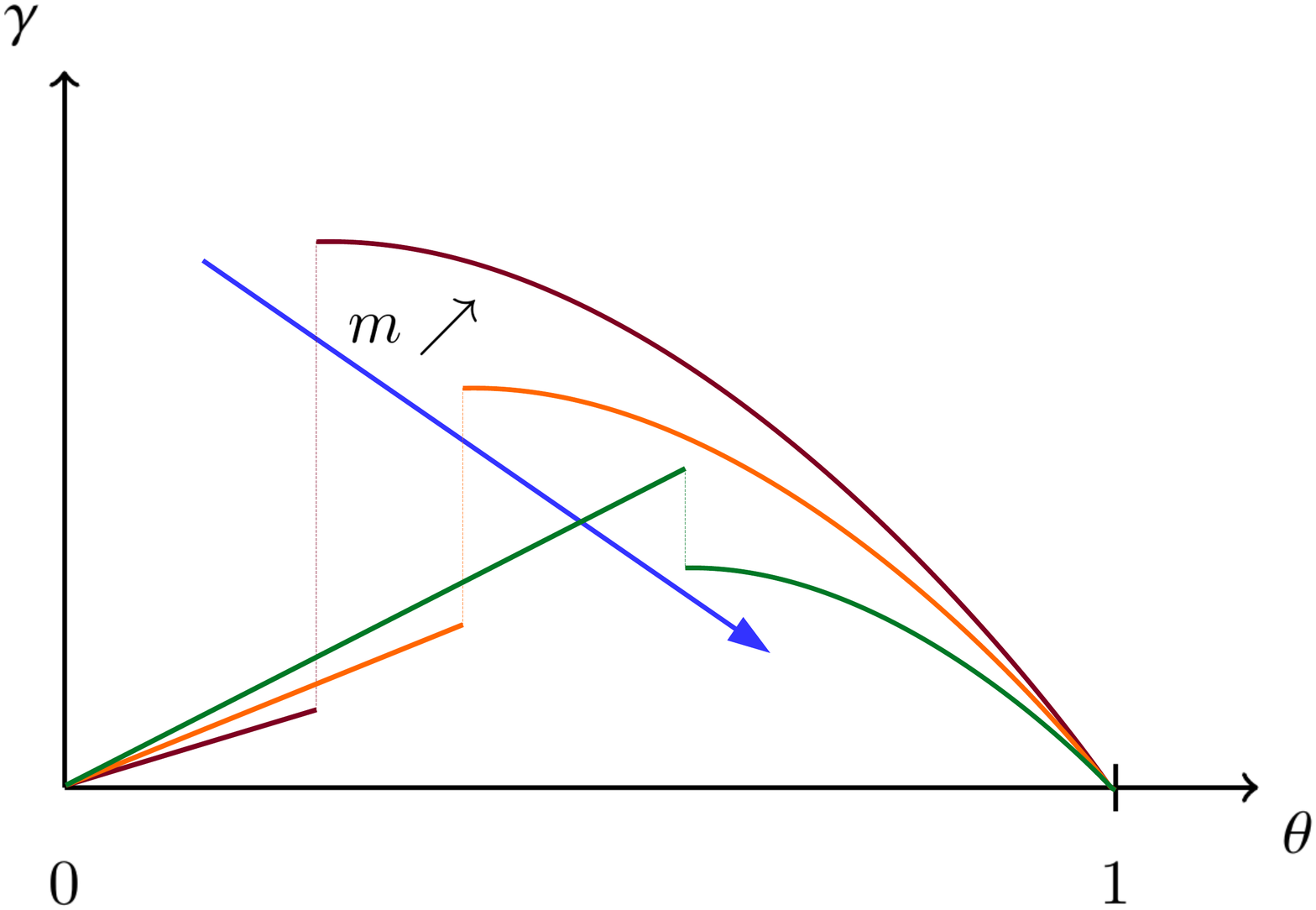}
\captionof{figure}{\label{fig:evoPortraitGlobal} Evolution of the phase portrait with $\massflux$.
Each curve corresponds to the solution curve for one value of $\massflux$.}
\end{minipage}
\end{figure}

\subsection{Existence of a solution}
We will now show that there exists at least one wave \modif{velocity}    
 $c<0$ such that the travelling wave problem previously stated has a solution for fixed values of $P$ and $T_0$.
We introduce $\Delta \gamma (c) = \gamma\left[\theta_s^+(c)\right] - \eta \gamma\left[\theta_s^-(c)\right]$, the \nomDG{} \modif{obtained for the regression velocity $c$} 
and $\xi(c) = \Delta \gamma (c) - S(c) $, which we will call the \nomXi{}. We introduce another assumption, non-restrictive for any real application:
\begin{assumption}
\label{assum:Qpyro_range}
The heat of the pyrolysis reaction $\Qpyro$ is such that $\Qpyro > -\Qgas$.
\end{assumption}

\begin{proposition}
\label{prop:existenceSolution}
Under Assumptions H\ref{assum:modif_pyro} and H\ref{assum:Qpyro_range},
there exists at least one wave \modif{velocity}    
 $c$ such that the problem stated in the reduced system 
can be solved.
All solutions for the wave \modif{velocity}    
 $c$ reside in the interval $(\cmax, 0)$, with $\cmax$ the dimensionless wave \modif{velocity} such that $\theta_s(\cmax)=1$.
 \end{proposition}

\begin{proof}
\noindent
The global phase portrait in the gas and solid phases is schematically represented in Figure \ref{fig:schemPhasePortrait}.
There is a jump of $\gamma$ at $\theta_s$, as explained in Remark \ref{remark:impactC}.
More precisely, the thermal boundary condition \eqref{eq:portrait:BC_interface} may be reformulated as:
$$
\Delta \gamma (c) = S(c)  ~~\Leftrightarrow~~ \xi(c) = 0
$$
with $\Delta\gamma$ the \nomDG{}, i.e. the dimensionless interface heat fluxes balance we obtain for a given value of $c$
by integrating equations \eqref{eq:portrait:solid} and \eqref{eq:portrait:gas} separately,
with boundary conditions \eqref{eq:portrait:BC_minf} and \eqref{eq:portrait:BC_pinf} respectively.
The \nomXi{} $\xi$ is non-zero when the interface thermal balance condition \eqref{eq:portrait:BC_interface} is not satisfied.
A complete solution to the reduced system may only be found if there exists a regression \modif{velocity} $c$ such that $\xi(c)=0$.

To prove the existence of such a value of $c$, we focus on the behaviour of $\xi$. We first aim at proving that $\xi$ is a continuous function of $c$.
To do so, we add $c$ as a variable in  our reduced system, subject to $\derivshort{c}{\theta}=0$ with the boundary condition $c(0)=c_0$ determined from
the pyrolysis law \eqref{eq:modifiedPyrolysisLaw}.
The reduced system in the solid phase can be recast to the following form:
\begin{center}
$\dfrac{d}{d\theta} {\begin{pmatrix} \gamma\\ c \end{pmatrix}} = f \begin{pmatrix} \gamma\\ c \end{pmatrix}$,
with $f \begin{pmatrix} \gamma\\ c \end{pmatrix} = \begin{pmatrix}  -c\\ 0\end{pmatrix}$
and initial conditions $\begin{pmatrix} \gamma(0)\\ c(0) \end{pmatrix} = \begin{pmatrix} \gamma_0\\ c_0 \end{pmatrix} = \begin{pmatrix} 0\\ -\massflux/\rho_s \end{pmatrix}$.
\end{center}

The associated flow is $\phi : (\theta; \gamma_0, c_0) \rightarrow (\gamma(\theta), c(\theta) )$. The theory of dynamical systems shows that, $f$ being here a 
$\mathcal{C}^\infty$ function, the flow is also $\mathcal{C}^\infty$ with respect to the initial conditions.
In particular, the solution profile for $\gamma$ in the solid phase depends continuously on $c_0=-\massflux/\rho_s$.
As we also assume (Assumption H\ref{assum:interface}) that the surface temperature $\theta_s$ is a $\mathcal{C}^\infty$ function of $c$,
$\gamma\left[\theta_s^-(c)\right]$ is $\mathcal{C}^\infty$ with respect to $c$.
The same reasoning can be applied to the gas phase for $\gamma\left[\theta_s^+(c)\right]$, so that $\Delta\gamma$ is $\mathcal{C}^\infty$.
$S$ is also trivially a $\mathcal{C}^\infty$ function of $c$. As $\xi$ is a sum of $\mathcal{C}^\infty$  functions of $c$, we conclude that
$\xi$ is $\mathcal{C}^\infty$ with respect to $c$.

Inspired by this property, we aim at finding two values of the wave \modif{velocity}    
 $c_1$ and $c_2$ such that $\xi(c_1)$ and $\xi(c_2)$ have opposite signs,
implying that there is at least one value of $c \in (c_1, c_2)$ such that $\xi(c)=0$.
We exhibit two limit cases for the wave \modif{velocity}    
 $c$, which naturally yield a different sign for $\xi$:
\begin{itemize}
 \item \textbf{Case $c=0$}: 
 In this case $\massflux=0$, i.e. the solid propellant remains inert. 
Using Assumption H\ref{assum:modif_pyro}, this equates to $\theta_s=0$. The temperature is uniform inside the solid phase.
Note that we are still satisfying the monotonicity property of $\theta$ shown in Proposition \ref{prop:monotonicity}
(the monotonicity is strict in the solid phase only if $\theta_s>0$). With $c=0$, Proposition \ref{prop:bilan:massflux} yields $u=0$.
Equations \eqref{eq:definition:S}, \eqref{eq:portrait:solid} and \eqref{eq:portrait:gas} lead to:
\mbox{$\gamma(\theta_s^-(0)) = 0$},
\mbox{$\gamma(\theta_s^+(0)) = ({ 2 \int_0^1 \Psi(y) dy })^{1/2} =  ({2 I_0})^{1/2}$ and $S(0) = 0$.}
Consequently $\Delta \gamma (0) = ({2 I_0})^{1/2} > S(0)$, and therefore $\xi(0) > 0$.

\item \textbf{Case $c=\cmax$}: 
The solution we are looking for is monotonous and thus requires $\theta_s \leq 1$.
Based on the pyrolysis law \eqref{eq:modifiedPyrolysisLaw}, the case $\theta_s=1$ corresponds to a certain value $\cmax<0$ of the wave \modif{velocity}.    
We can then directly integrate the reduced system equations \eqref{eq:portrait:solid} and \eqref{eq:portrait:gas} to obtain:
\mbox{$\gamma(\theta_s^-) = -\cmax$,}
\mbox{$\gamma(\theta_s^+) = 0$}
and 
\mbox{$S(c) = \eta\,\cmax\,\Qpyro/ (\Qpyro + \Qgas) $.}
Thus, $\xi(\cmax) = \Delta\gamma(\cmax) - S(\cmax) = \eta \cmax \left( 1 - \Qpyro/(\Qpyro+\Qgas) \right) = \eta\,\cmax\,\Qgas/(\Qpyro + \Qgas)$.
Assuming  H\ref{assum:Qpyro_range}, we obtain $\xi(\cmax) < 0$.
\end{itemize}

In realistic cases for $\Qpyro$, we have shown that $\xi(0) > 0$ and $\xi(\cmax) < 0$. Therefore, as $\xi$ is a continuous function of $c$,
there exists at least one value of $c \in (\cmax, 0)$ such that $\xi(c) = 0$.
Potential solutions with $c>0$ or $c<\cmax$ are physically meaningless and are not further considered.
The existence of a solution for the reduced system implies that a solution also exists for the dimensionless system and the simplified travelling wave system.
\end{proof}

\subsection{Uniqueness of the solution}

Having proved that there exists at least one value of $c$ such that the travelling wave problem can be solved, we now proceed to show that there is only one such value.
There are two cases, depending on the sign of $\Qpyro$.

\begin{proposition}
\label{prop:uniquenessQp_negatif}
  If $~\Qpyro<0$, there exists a unique value of the wave \modif{velocity}    
 $c$ such that the reduced system has a solution.
\end{proposition}

\begin{proof}
\label{proof:uniquenessQp_negatif}
\noindent
Studying the existence of a solution in Proposition \ref{prop:existenceSolution}, we have introduced the \nomXi{} $\xi(c) = \Delta\gamma(c) - S(c)$.
A solution to the reduced system with regression \modif{velocity}    
 $c$  only exists if $\xi(c)=0$.
We have shown that $\xi$ undergoes a change of sign between $c=0$ and $c=\cmax$. This implies that there exists at least one value of $c$ such that $\xi(c)=0$.
As we aim at proving that there is only one such value of $c$, we need to show that $\xi$ is a monotonous function of $c$.
To do so, we will study separately the evolution of the two terms appearing in the definition of $\xi$: $\Delta\gamma$ and $S$.
For improved readability, we introduce $\gamma^- = \gamma(\theta_s^-, c)$ and $\gamma^+ = \gamma(\theta_s^+, c)$.

\subparagraph*{\it Evolution of $\Delta\gamma$}
We have $\Delta\gamma = \gamma^+ - \eta\gamma^-$. To study the evolution $\Delta\gamma$,
we will first analyse the behaviour of $\gamma^-$ and $\gamma^+$.

In the solid phase, we have seen that we may solve equation \eqref{eq:portrait:solid} analytically and find $\gamma^- = - c\theta_s(c)$.
Deriving with respect to $c$, we obtain:
$$\derivshort{\gamma^-}{c} = -\theta_s(c) - c\partialdershort{\theta_s(c)}{c}$$

The Assumption H\ref{assum:interface} on the pyrolysis law \eqref{eq:modifiedPyrolysisLaw} implies $\derivshort{\massflux}{T_s} > 0$, hence
$\derivshort{T_s}{\massflux} > 0$ and consequently $\derivshort{\theta_s}{c} < 0$.
Therefore $\partialdershort{\gamma^-}{c} < 0$, i.e. that is the more $c$ diminishes ($\massflux$ increases),
the more thermal power is needed to maintain the solid phase temperature profile, as we would expect.

We now focus on the evolution of $\gamma^+$.
In the gas phase, the dimensionless temperature gradient at the interface is given by integrating equation \eqref{eq:portrait:gas}:
$$\gamma^+ = \int_1^{\theta_s(c)} \derivshort{\gamma}{\theta} ~ d\theta
	   = \int_1^{\theta_s(c)} \left( \dfrac{-\Psi(\theta)}{\gamma(\theta)} - \eta c \right) d\theta
$$
Deriving this expression with respect to $c$ yields:
$$\derivshort{\gamma^+}{c} =
    \underbrace{\int_1^{\theta_s(c)} \derivshort{ ( \derivshort{\gamma}{\theta} ) }{c} d\theta}_{A = \left( \partialder{\gamma^{\modif{+}}}{c}\right)_{\theta_s} (\theta_s, c)}
    ~+~
    \underbrace{\left( \dfrac{-\Psi(\theta_s)}{\gamma^+} - \eta c \right) \derivshort{\theta_s}{c}}_{B = \left(\partialder{\gamma^{\modif{+}}}{c}\right)_{c} (\theta_s, c)}
$$   

Let us study the sign of $A$ and $B$.
The term $A$ is the derivative of $\gamma^+$ with respect to $c$ at constant $\theta_s$.
Its sign may be found by following the same reasoning as Zeldovich in his work on laminar flames \cite{livreZeldovich}, which we reproduce hereafter for the
sake of completeness.
Deriving equation \eqref{eq:portrait:gas} with respect to $c$, we obtain:
\newcommand{\tempVar}{\left( \dfrac{\gamma}{-\eta} \right)}
$\partialdershort{( \partialdershort{\gamma}{\theta} )}{c} =
-\eta + ({\partialdershort{\gamma}{c}}/{\gamma^2})\Psi(\theta)$.
\renewcommand{\tempVar}{ y } 
Introducing $y=-{\gamma}/{\eta}$ and $\Pi(\theta) = {\Psi(\theta)}/{\eta^2}$ yields
\mbox{$\partialdershort{( \partialdershort{ \tempVar }{\theta} )}{c} =
1 + ({\partialdershort{ \tempVar}{c}}/{ \tempVar^2}) \Pi(\theta)$.}
Zeldovich has shown (\cite{livreZeldovich}, page 256) that the solution to this equation is:
$\quad \partialdershort{y}{c}(\theta) = - \exp(\rchi(\theta)) \int_{\theta}^1 \exp(-\rchi(z)) dz$, with $\rchi = \int {\Pi}/{y^2}$.
Therefore $\derivshort{y}{c}(\theta)<0, ~\forall \theta \in [0,1]$
and consequently \mbox{$A = -\eta \partialdershort{y}{c}(\theta_s) > 0$}.

Let us now determine the sign of $B$.
Based on Assumption H\ref{assum:gasphase}, we have $\Psi \geq 0$.
The monotonicity of the temperature profile in the gas phase implies $\gamma^+>0$, and Assumption H\ref{assum:interface} leads to $\derivshort{\theta_s}{c} < 0$.
Therefore $B$ is positive only if $\Psi(\theta_s)/{\gamma^+} > - \eta c$, which may not always be true,
thus we cannot directly conclude on the sign of $\derivshort{\gamma^+}{c}$.

However, if we combine the derivatives of $\gamma^+$ and $\gamma^-$ to express the derivative of $\Delta\gamma$, the terms containing $\eta c$ cancel out:
$$\derivshort{\Delta\gamma}{c} = \derivshort{\gamma^+}{c} - \eta \derivshort{\gamma^-}{c}
  = A  + \left( \dfrac{-\Psi(\theta_s)}{\gamma^+} -\eta c \right) \derivshort{\theta_s}{c}
    - \eta ( -\theta_s - c \derivshort{\theta_s}{c})$$
$$ \Rightarrow \derivshort{\Delta\gamma}{c}  = A + \eta \theta_s - \dfrac{\Psi(\theta_s)}{\gamma^+} \derivshort{\theta_s}{c}$$

Overall, the three remaining terms are positive, hence $\derivshort{\Delta\gamma}{c} > 0$.

\subparagraph*{\it Evolution of $S$}
Deriving equation \eqref{eq:definition:S}, we get: $\quad \partialdershort{S}{c} = \eta {\Qpyro}/({\Qpyro + \Qgas})$.
Following Assumption H\ref{assum:Qpyro_range} ($\Qpyro > -\Qgas$), we conclude that $\partialdershort{S}{c}$ has the same sign as $\Qpyro$.
In this proposition, we assume \mbox{$\Qpyro \in (-\Qgas, 0]$}, therefore we obtain $\partialdershort{S}{c} < 0$.

\subparagraph*{\it Evolution of $\xi$}
We now have determined the signs of each term appearing in the derivative of the \nomXi{} $\xi$ with respect to $c$.
Using the previously established relations, we can write:
$$\partialdershort{\xi}{c} = \partialdershort{\Delta\gamma}{c} - \partialdershort{S}{c} > 0$$
We conclude that $\xi$ is a monotonous function of $c$.
We have shown in the proof of Proposition \ref{prop:existenceSolution} that $\xi(\cmax) < 0$ and $\xi(0) > 0$,
i.e. that there exists at least one solution wave \modif{velocity}    
 $c$ such that the reduced system is solved.
The monotonicity of $\xi$ we just established is the additional property needed to prove that there is only one such solution.
Physical interpretations of the behaviour of $\xi$, $\Delta\gamma$ and $S$ are \modif{given} in \ref{part:physical_interpretation}.
\end{proof}

\begin{proposition}
\label{prop:uniquenessQp_positif}
 If $\Qpyro>0$, there exists a unique value of the wave \modif{velocity}    
 $c$ such that the problem stated in the reduced system 
 can be solved. This solution $c$ belongs to the interval $(\cmax, \cmin)$ with $\cmin$ such that $\theta_s(\cmin) = {\Qpyro}/({\Qgas+\Qpyro})$.
\end{proposition}

\begin{proof}
\label{proof:uniquenessQp_positif}
This result is obtained in a manner almost identical to the previous one. The difference lies in the behaviour of $S$.
With $\Qpyro>0$, we have $\partialdershort{S}{c} > 0$, as is $\partialdershort{\Delta\gamma}{c}$, therefore we cannot directly conclude on the sign of 
$\partialdershort{\xi}{c}$ for $c \in (\cmax,0)$.
To circumvent this difficulty, we will show that there exists a value $\cmin$ such that we always have $c < \cmin$,
which verifies $\xi(\cmin) > 0$, and such that $\xi$ is monotonous on the interval $(\cmax, \cmin)$.
Starting from the relations established in the proof of Proposition \ref{prop:uniquenessQp_negatif}, we can express $\derivshort{\xi}{c}$:

  $$ \derivshort{\xi}{c}  = \derivshort{\Delta\gamma}{c} - \derivshort{S}{c} = \underbrace{A - \dfrac{\Psi(\theta_s)}{\gamma^+} \derivshort{\theta_s}{c}}_{>0}
  + \eta \left( \theta_s - \dfrac{\Qpyro}{\Qgas	+ \Qpyro} \right)$$

Consequently, to ensure $\derivshort{\xi}{c}  > 0$, it is sufficient that the last term is positive:
$$\theta_s > \thetamin = \dfrac{\Qpyro}{\Qpyro+\Qgas} = \dfrac{\Qpyro}{c_p (T_f-T_0)}
\quad\Leftrightarrow\quad
T_s > \Tslim = 
T_0 + \dfrac{\Qpyro}{c_p}$$

Here we can give a physical interpretation of $\Tslim$. It is the temperature that would be achieved at the interface without any heat feedback
from the gas phase. Indeed if $\derivshort{T}{x}(0^+)=0$, we can integrate equation \eqref{eq:onde:solid} from $-\infty$ to $0$
and find $T_s = T_0 + {\Qpyro}/{c_s}$. Following Assumption \ref{assum:cpcs} ($c_p=c_s$), we recover our previous expression of $\Tslim$.

Now we need to show that all acceptable solutions have the property $T_s > \Tslim$. The monotonicity of the temperature in the gas phase, established in Proposition
\ref{prop:monotonicity}, associated with the condition $T_s < T_f$ shows that $\derivshort{T}{x} > 0$ in the gas phase.
This means that $\gamma$ is always positive in the gas phase: heat is always conducted from the gas phase into the solid phase.
As a consequence, $T_s > \Tslim$ is always satisfied in our problem.
That is also what we would expect from a physical point of view, as we know the gas phase will actually heat up the solid, not cool it down.
Moreover, using the constant combustion enthalpy property from Proposition \ref{prop:bilan:enthalpy} and the global energy balance from Proposition \ref{prop:bilan:T},
we find that $T_s>\Tslim$ is also the required condition to ensure $Y(0^+)<1$.

Overall, we are now assured that the surface temperature $\theta_s$ will always be higher than $\theta_{s,min}$.
Via the pyrolysis law \eqref{eq:modifiedPyrolysisLaw}, the minimum surface temperature $\theta_{s,min}$ corresponds to a regression \modif{velocity}    
 $\cmin<0$.
Therefore we conclude that $c$ will always be lower than to $\cmin$. We remind the reader that $c<0$, therefore $c<\cmin$ yields $|c|>|\cmin|$ (faster regression).

Let us now compute the value of $\xi$ for this value of $c$:
$$\xi(\cmin) = \gamma(\thetamin^+, \cmin) - \eta\gamma(\thetamin^-, \cmin) - \eta \cmin /(1+k) = \gamma(\thetamin^+, \cmin)$$
The strict monotonicity of $\theta$ implies that $\gamma$ is always positive. As a consequence, $\xi(\cmin) > 0$.
We have shown in the proof of Proposition \ref{prop:existenceSolution} that $\xi(\cmax) < 0$, therefore there exists a solution
wave \modif{velocity}    
 $c$ in the interval $(\cmax, \cmin)$, such that $\xi(c)=0$.

On this interval, we have established that $\derivshort{\xi}{c} >  0$, whence we conclude that the solution is unique within this interval.
Let us underline again that solutions outside of this interval are not physical and would lead to a violation of the monotonicity of the temperature profile.
Physical interpretations of the behaviour of $\xi$, $\Delta\gamma$ and $S$ are presented in \ref{part:physical_interpretation}.
\end{proof}

At this point, we have proved that there exists a unique solution to the reduced system,
therefore also for the dimensionless one and for the simplified travelling wave problem which were presented in \ref{part:derivation_modele_base}.
There exists only one steady travelling combustion wave solution for the burning of a homogeneous solid propellant with
simplified kinetics and a pyrolysis concentrated at the surface, with the surface temperature being linked to the mass flow rate by a pyrolysis law
such that the mass flow rate monotonously increases with the surface temperature.
The proof is valid for a very wide range of values for the heat of reaction of the pyrolysis process, and for any value of the gas phase activation energy.
\begin{remark}
\label{remark:difficulty_cpcs}
The Assumption H\ref{assum:cpcs} ($c_s=c_p$) made the proof of uniqueness much easier. If we had not used it, the \nomS{} $S$ would have a more complex variation
with respect to $c$ and no easy conclusion on uniqueness would be possible.
However the assumption can easily be relaxed in the numerical method presented further in this paper, as it only changes
the definition of $S$, $\Qpyro$ and $T_f$ (see Section \ref{subsubsection:parametric_study_cpcs}).
It is likely that the solution may remain unique on a certain range of values for the ratio $c_p/c_s$, and this was indeed
observed in our test case for a wide variety of values for this ratio.
\end{remark}

\begin{remark}
 Johnson and Nachbar \cite{1963JohnsonNachbarUnicity} proved the uniqueness of $c$ for any fixed value of $T_s$.
 This study case can also be treated with the approach we have presented, however a few adjustments are necessary, which is why we expose the main
 steps in appendix \ref{remark:casParticulierTsConstante}.
\end{remark}
\begin{remark} 
 The study presented in this paper also encompasses the laminar flame study by Zeldovich \cite{livreZeldovich}. In this case, $\theta_s = \theta_{ignition}$ is a constant,
 and $S=0$ as no chemical reactions takes place at the interface, therefore the existence and uniqueness is proved directly from Proposition \ref{prop:uniquenessQp_negatif}.
\end{remark}

\subsection{Heteroclinic orbit and critical points}
\label{subsection:point_crit}
The points $x=-\infty$ and $x=+\infty$ are critical points for the dimensionless system
, i.e. all derivatives are zero.
These points correspond to $(\theta=0, \gamma=0)$ and $(\theta=1, \gamma=0)$.
The dynamics of the dimensionless system in phase space is a heteroclinic orbit that joins these two points.
This orbit and the associated treatment of the critical points is very similar to the bistable planar waves studied in \cite{Volpert3}.
The first critical point $(0,0)$ is more difficult to analyse, as it is not a hyperbolic point, however we can easily integrate \eqref{eq:portrait:solid}
and find that the solution is $\gamma=-c\theta$.
The other critical point $(1,0)$, in the gas phase, is a hyperbolic point, therefore the solution curve (orbit) will depart from the associated stable manifold.
We can then determine the slope $\derivshort{\gamma}{\theta}(1)$ by means of a linearisation. We use the approximations 
$\gamma(\theta) = \alpha(\theta-1)$ and $\Psi(\theta) = \beta(\theta-1)$, with  $\alpha = \derivshort{\gamma}{\theta}(1)$ and $\beta = \derivshort{\Psi}{\theta}(1)$.
Following remark \ref{remark:psi}, we know that $\beta<0$. Injecting these linearised expressions into \eqref{eq:portrait:gas}, we get:
$\alpha^2 +  \alpha\tilde{\massflux} + \beta = 0$.
This second order equation has two real solutions of opposite signs.
As we require $\alpha = \derivshort{\gamma}{\theta}(1) < 0$ so that our solution remains in the half-plane $\gamma \geq 0$, we find:
\begin{equation}
\label{eq:slope_critic}
\alpha  = \frac{\tilde{\massflux}}{2} \left( 1 - \sqrt{ 1 - \dfrac{4 \beta}{\tilde{\massflux}^2} }   \right)
\end{equation}
The behaviour around the two critical points will be used in the numerical strategy based on a shooting method to integrate the dynamics of the orbit.

\subsection{Physical interpretation and discussion}
\label{part:physical_interpretation}
It will be easier to interpret the behaviour of our system by considering its variations with respect to the mass flow rate $\massflux$.
Let us remind the reader that the pyrolysis mass flow rate $\massflux=-\rho_s c$ is positive, whereas $c$ is negative.
Consequently and for any variable $q$:
$$\partialdershort{q}{c}<0 \Leftrightarrow \partialdershort{q}{\massflux}>0$$
In the present paper, we have introduced $\Delta\gamma$, which the dimensionless thermal power excess that is available to power the pyrolysis process, i.e. gas heat feedback minus the thermal power
used to heat up the solid, and $S$, the dimensionless thermal power that is required for the pyrolysis process to be sustained at the given value of $\massflux$.
We have shown that $\derivshort{\Delta\gamma}{\massflux} < 0$, therefore increasing $\massflux$ will decrease the thermal power available for the pyrolysis.
The sign of $\derivshort{S}{\massflux}$ indicates how the required thermal power evolves with the pyrolysis mass flow rate $\massflux$.

\paragraph*{case $\mathbf{\Qpyro<0}$}
In the case $\Qpyro<0$, the pyrolysis process is endothermic, i.e. it absorbs heat from the gas and solid phases. This can be the case if the sublimation of $\Gis$ into
$\Gig$ is very demanding in terms of energy, which corresponds to $L_v>Q_s$ in remark \ref{remark:Qpyro}.
We showed that in this case, $\derivshort{S}{\massflux}>0$.
If, for an arbitrarily chosen value of $\massflux$, we have $\Delta\gamma < S$, it means that the heat feedback from
the gas phase is too low compared to the heat that would be absorbed by the solid phase and the pyrolysis reaction in a stationary state.
The fact that $\derivshort{\Delta\gamma}{\massflux}<0$ and $\derivshort{S}{\massflux}>0$ shows that as we  lower the mass flow rate, the thermal power excess
transmitted by the gas phase to the interface increases whereas the thermal power needed for the pyrolysis decreases.
As we have seen that in the limit $\massflux \rightarrow 0$, $\Delta\gamma > S$, we know that we will find one value of $\massflux$ such that both powers cancel out.
Conversely if we start with $\massflux$ such that $\Delta\gamma > S$, we need to increase the value of $\massflux$. The limit case $\massflux = \massflux(T_s = T_f)$
yields $\Delta\gamma < S$, therefore we are also ensured that we will find one solution for $\massflux$.

\paragraph*{case $\mathbf{\Qpyro>0}$}
The same reasoning can be applied.
In this case the pyrolysis is exothermic, thus it also contributes to the heating of the solid phase.
We showed that $\derivshort{S}{\massflux}<0$, i.e. the thermal power required by the pyrolysis decreases as the mass flow rate increases, in the sense that it
is actually negative and increasing in magnitude. This is physically coherent with the fact that the pyrolysis is exothermic.
We have established that in the interval $(\cmax, \cmin)$, $\derivshort{\xi}{c}>0$, i.e. $\derivshort{\xi}{\massflux}<0$ for $\massflux \in (0, \massflux(T_f))$.
It shows that as we increase the mass flow rate $\massflux$, the thermal power excess transmitted by the gas phase to the interface
will decrease more rapidly than the thermal power needed for the pyrolysis.
Therefore, starting from a value of $\massflux$ such that the heat feedback is too strong ($\xi>0$), lowering $\massflux$ will only worsen the interface thermal balance.
We actually need to increase $\massflux$, up until the point where the thermal power $S$ required by the pyrolysis catches up with the thermal power excess $\Delta\gamma$.\\

\noindent The gradient jump $\left[ \derivshort{\theta}{x} \right]_{0^-}^{0^+}$ is the same as  $\left[ \gamma \right]_{\theta_s(c)^-}^{\theta_s(c)^+}$.
Using \eqref{eq:adim:BCs:bilanInterface}, we can  rewrite this as 
$\left[ \gamma \right]_{\theta_s(c)^-}^{\theta_s(c)^+} = S(c) + (\eta-1) \gamma\left( \theta_s(c)^- \right)$.
In the particular case where $\eta=1$, i.e. both phases have the same thermal conductivity, this reduces to $S(c)$, thus the gradient jump has the sign of $S$.
If we have $\eta \neq 1$, the sign of the gradient jump will depend on the gradient in the solid phase at the interface.
For example, if $\eta>1$, the temperature gradient jump at the interface will be positive only if $S(c) > (1-\eta)\gamma(\theta_s(c)^-)$.

This theoretical study brings in two aspects. First, it allows to describe a greater variety of physical scenarios, compared to the ones represented by the existing
analytical models.
Second, and this is the purpose of the next section, it allows for an efficient numerical resolution.

\section{Numerical method and \modif{verification} against a CFD code}
\label{section:numCFD}
\modif{We now explain how the previous analysis is used to construct a numerical shooting method to iteratively determine the solution, i.e. the correct wave velocity (eigenvalue)    
 and temperature profile (eigenfunction).
We also present a one-dimensional CFD code developed at ONERA for the study of solid propellant combustion. 
This CFD code can be adapted to use the same level of modelling as the shooting method.
Our first objective is therefore to compare the results of both methods to cross-verify these tools within the framework defined by our modelling assumptions
(Sections
\ref{subsubsection:refcase_unitLewis} and \ref{subsubsection:parametric_study_Ea}).
Our second objective is to show how the shooting method may be extended to relax some of these assumptions and what limitations may be encountered
(Sections \ref{subsubsection:parametric_study_cpcs} and \ref{subsubsection:parametric_study_Lewis}).
For this purpose, the CFD code will serve as a reference, as it allows for a straightforward relaxation of several assumptions.}

\subsection{Shooting method}
\subsubsection{Determination of the phase portrait for a given $c$}
For a given value of $c$, we can integrate the dimensionless equations \eqref{eq:portrait:solid} and \eqref{eq:portrait:gas} from the reduced system,
which are first-order ODEs for the variable $\gamma$ as a function of $\theta$.
In the solid phase, the integration is analytical, as we directly obtain $\gamma(\theta) = -c\theta$. This gives us the value of $\gamma$
for $\theta \in [0, \theta_s(c)]$.
In the gas phase, equation \eqref{eq:portrait:gas} can be written as: $\derivshort{\gamma}{\theta} = -\eta c - (\Psi/\gamma)$.
We need to integrate this equation from $\theta=1$ to $\theta=\theta_s$. 
As explained in subsection \ref{subsection:point_crit}, the starting point $(\theta=1, \gamma=0)$ is a critical point for our system, therefore starting a numerical
integration from this point is impossible. To overcome this problem, we simply use the linearised solution slope $\alpha$ given in \eqref{eq:slope_critic},
and start the integration from $( 1-\Delta\theta, -\alpha \Delta\theta)$, avoiding the critical point. We typically use $\Delta\theta = 10^{-6}$.
To maximise the accuracy, the integration of the gas phase equation \eqref{eq:portrait:gas} is performed using the Radau5 algorithm \cite{hairer_radau5},
featuring an adaptive step size, with very tight tolerances ($\approx 10^{-14}$).
Once the profile of $\gamma$ is computed, we can go back to the spatial representation by using the definition $\gamma = \derivshort{\theta}{x}$ to compute
$x(\theta) = \int_{\theta_s}^{\theta} \frac{z}{\gamma(z)} dz$. This formula also ensures that $x(\theta_s) = 0$.

\subsubsection{Determination of $c$ through a dichotomy process}
Based on our analysis of $\xi$, we know that $\xi(0)>0$ and $\xi(\cmax)<0$. In the case $\Qpyro<0$,
we start a dichotomy from the two initial points $0$ and $\cmax$, the latter being computed beforehand from the global energy balance and the pyrolysis law.
If $\Qpyro>0$, we replace the starting value $0$ with $\cmin$.
In both cases, $\xi$ is monotonous between the two initial points and undergoes a change of sign, therefore convergence of the dichotomy process is ensured.
For each new guess of $c$, we integrate the reduced system as explained previously,
and obtain the value of $\Delta\gamma(c)$. We compare it to the value of $S(c)$ to compute $\xi(c)$. Based on the sign of $\xi(c)$, we can shrink the interval
where $\xi$ changes sign, until the change in $c$ between each iteration becomes small enough.

We could also perform a constrained optimisation on the variable $c$, minimising the objective function
$f(c) = [ S(c) - \Delta\gamma(c)]^2$, with the constraint $c \in [\cmax, c_{min}]$.
Practically, the optimisation method is quicker to find the approximate solution, but fails at determining $c$ as precisely as the dichotomy process,
even when using tight tolerances. However the dichotomy requires many iterations, therefore the more advanced Brent root-finding method \cite{Brent1973} was used.
In our test cases, the solution was usually found within 10 iterations.

\begin{remark}
 This \modif{semi-analytical} method is bound to be more accurate than the analytical models discussed in the introduction, as these models basically use the same assumptions, but also
 assume that the activation energy of the gas phase reaction is either very high or zero. Our method does not need this information and will better reproduce
 the gas flow. This comes at the cost of having to iterate on the value of $c$, each time integrating numerically the reduced system.
 However, this cost will be rather low, as each iteration only requires the integration of the simple ODE \eqref{eq:portrait:gas}.
 This method is consequently especially useful to perform extensive parametric studies.
\end{remark}

\subsubsection{Error of the method}
The numerical shooting method contains 3 sources of error:
\begin{packed_list}
 \item error in the estimation of $\derivshort{\gamma}{\theta}(1)$, used to avoid the critical point in the gas phase;
 \item error in the numerical integration of the gas phase temperature profile;
 \item convergence precision achieved by the shooting method on the value of $c$.
\end{packed_list}
Let us address the different items in this list.
First, a simple parametric study on the value of $\Delta\theta$ has shown that $\derivshort{\gamma}{\theta}$ is a constant
in the neighbourhood of the critical point. Different values of $\Delta\theta$ have
been tested and the converged regression \modif{velocities}    
are exactly identical for all $\Delta\theta$ lower than $10^{-3}$. Consequently the linearisation around the critical point
is a reasonable approach and the error it produces is zero up to machine precision.
The numerical integration of the gas phase with the Radau5 algorithm with very tight tolerances is also close to machine-precision,
as the algorithm is of order 5 and the step size is dynamically adapted to maximise accuracy.
A convergence study has been performed by varying the integration tolerance from $10^{-3}$ to $10^{-15}$, each time determining the solution $c$
via the dichotomy process (Brent method).
It has been found that the solution wave \modif{velocity}    
 obtained for a tolerance of $10^{-14}$ is converged with a relative error of $10^{-14}$.
Finally, the dichotomy process usually is able to converge the solution $c$ with a relative error of the order of $10^{-15}$.
Overall the only practical limitation to the precision of this numerical shooting method is the machine accuracy chosen for implementation.
All our numerical computations have been performed with double precision.

\subsection{Reference CFD code}
We wish to compare our semi-analytical model with a proven CFD code in a less restrictive framework. The aim is to verify the shooting method results and validate our
assumptions.
The CFD tool developed at ONERA is a Fortran90 code based on a finite-volume approach for the one-dimensional problem, inspired from  \cite{Smooke_Gio}. The model has also been adapted for the \modif{combustion of aluminium particles} \cite{dmitri_modele1D}. The molecular diffusion fluxes are approximated using a second-order central difference
scheme. The convective fluxes are approximated \modif{either} by a first-order upwind scheme, a second-order hybrid scheme weighted by the local Péclet number, \modif{or a second-order centred scheme}.
The equations are written in their steady form in the travelling combustion wave reference frame.
\modif{These equations are discretised in space and, together with the boundary conditions,} they represent a system of coupled non-linear equations.
A modified Newton method \modif{with damping} is used to determine the solution, as described in \cite{Smooke_Gio}. The Jacobian matrix is computed numerically \modif{by finite differences}.
The convergence strongly depends on the initial state. If convergence is poor, temporal evolution terms can be added
to the equations to approach the steady-state solution through a number of transient iterations.
This code also contains an automatic grid-refinement algorithm that ensures the mesh is fine enough
in the regions where the gradient or the curvature of the solution variables are high.
The refinement is performed after each successful convergence to a steady solution, until all refinement criteria are satisfied.
The code can handle detailed chemistry by accessing reaction and thermodynamic data through an interface with CHEMKIN-II \cite{ref_CHEMKIN2}.
Detailed molecular transport with binary species diffusion is also possible with the use of the EGlib library \cite{EGlib}.
However for the comparison with the numerical shooting method, these additional capabilities are not used.
This CFD code yields solutions which are subject to different sources of errors: the quality of the discretisation (grid refinement),
the tolerance for the Newton method, and the fluxes approximations. All CFD results presented further on were computed with automatic mesh refinement criteria
such that any additional refinement does not change the solution. The tolerance on the \modif{norm of the Newton step is $10^{-8}$,
and it was verified that lowering this tolerance did not change the results.}

\subsection{Numerical verification and parametric studies}

\subsubsection{Reference case with unitary Lewis} 
\label{subsubsection:refcase_unitLewis}

\modif{The reference case we will use throughout the rest of the article is}
the combustion of a one-dimensional equivalent of the AP-HTPB-Al propellant. The values for the different
properties are adapted from \cite{lengelle}.
The reaction rate is $\tauxreaction{} = A [\Gi] T \exp(-T_a/T)$ with $[\Gi]=\rho Y / \molarmass$ the molar concentration of the species $\Gi$.
The activation energy for the gas-phase reaction is $E_a = 58.7~ \textnormal{kJ.mol}^{-1}$, which corresponds
to an activation temperature $T_a = E_a/R = 7216$ K.
The specific heats are $c_s = c_p = 1.2 \times 10^3 \textnormal{J.kg}^{-1}\textnormal{K}^{-1}$.

The pressure is set to 5 MPa. The heat of reactions are $\Qpyro = 1.8 \times 10^5 \textnormal{ J.kg}^{-1}$, and $\Qgas = 3.9 \times 10^6 \textnormal{ J.kg}^{-1}$. 
For the CFD code, the diffusion coefficient $D_g$ for both species is taken as a linear function of $T$, such that the Lewis number
$\mathrm{Le} = (\lambda_g / (\rho (T) c_p))/ D_g(T) $ is 1 across the gas phase.
Figure \ref{fig:comparaison_Dmitry_50bars} shows a graphical comparison of the dimensionless temperature $\theta$ and mass fraction $Y$ profiles.
The agreement is very good, and has been verified for several other values of the pressure $P$ (e.g. 0.5 MPa). The relative error between the regression speed obtained from the semi-analytical tool and the one obtained with the CFD code with mesh adaptation is around $10^{-7}$. This allows us to conclude on the verification of our numerical strategy and model implementation.

\begin{figure}
\centering
\begin{minipage}{.47\textwidth}
  \centering
    \includegraphics[width=1.1\textwidth]{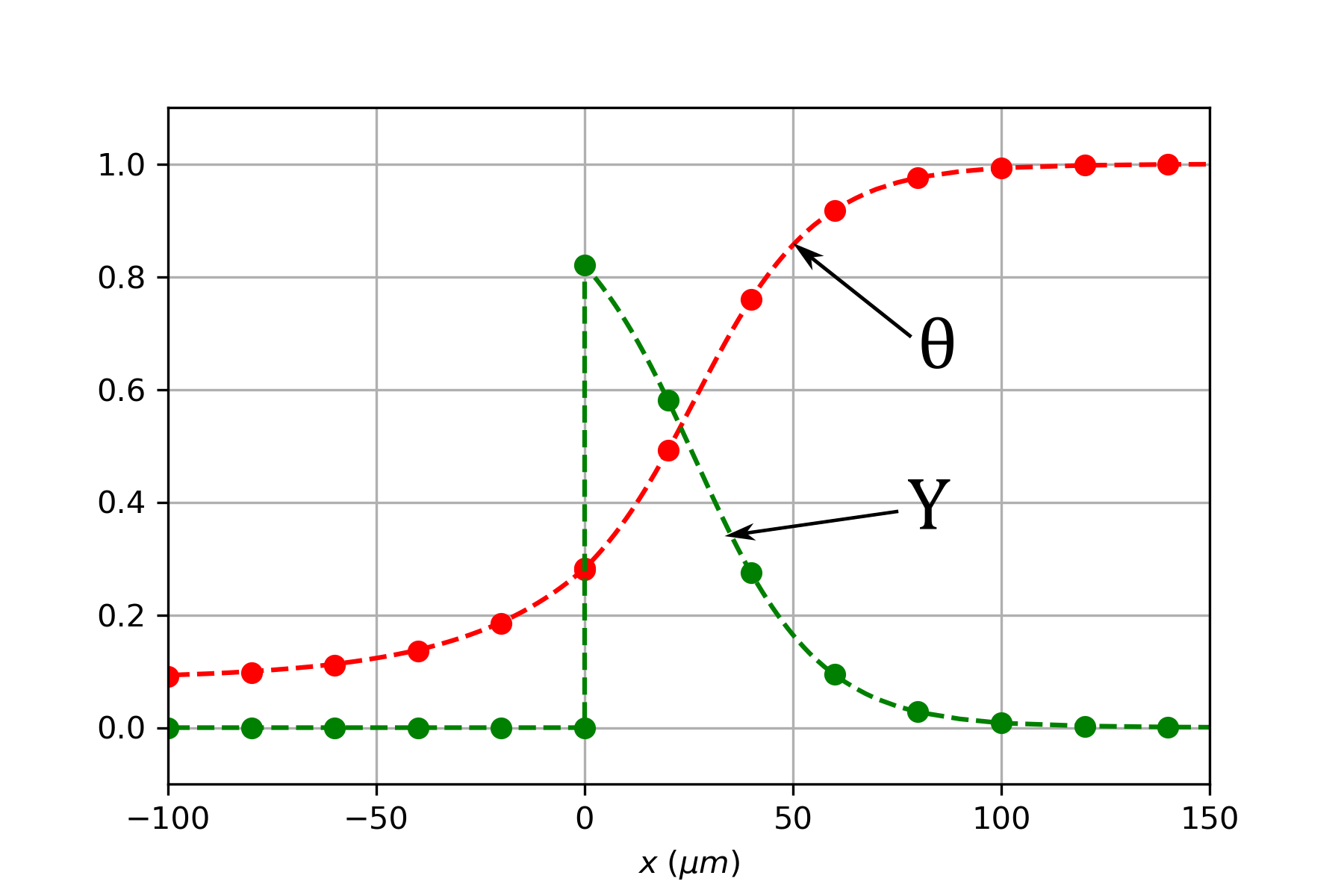}
    \captionof{figure}{Plot of $\theta$ and $Y$ of the semi-analytical solution (dots)
     with the numerical simulation (dashed lines)}
    \label{fig:comparaison_Dmitry_50bars}
\end{minipage}%
\hfill
\begin{minipage}{.47\textwidth}
	\centering
	  \includegraphics[width=1.1\linewidth, trim={0cm 0cm 0 0cm},clip]{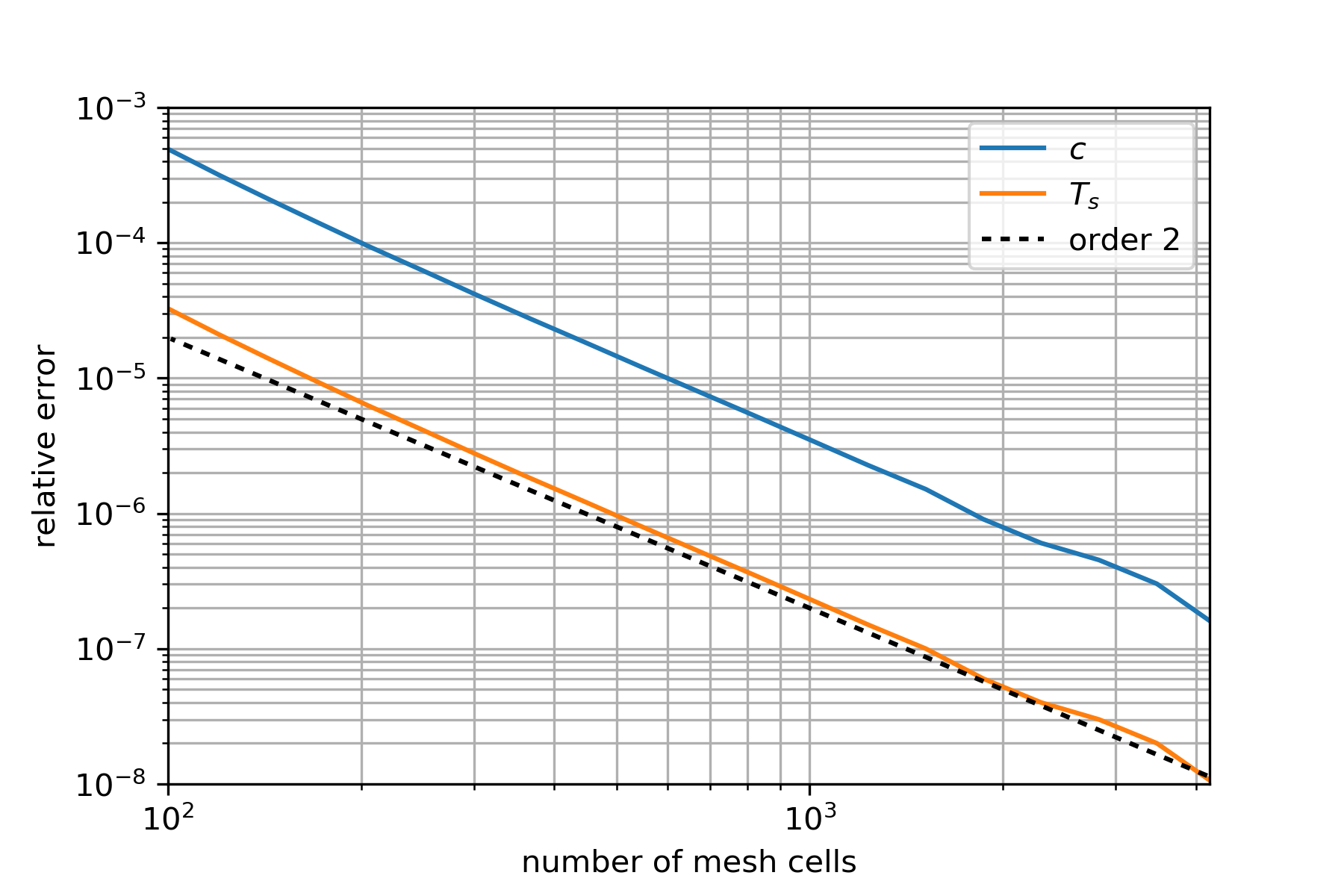}
	  \captionof{figure}{Convergence of the CFD solution towards the semi-analytical solution with an adapted mesh}
	  \label{fig:convergence_CFD_vers_semi}
\end{minipage}
\end{figure}

\subsubsection{Spatial convergence of the CFD solution}
\label{subsubsection:CFD_convergence}
\modif{It is interesting to study the convergence in space of the steady-state CFD solution.
We perform multiple simulations on increasingly refined grids. The meshes are generated as follows: knowing the temperature profile from the semi-analytical solution and starting from an initial grid point at $x=0$ (interface), the other grid points are placed such that the difference in interpolated temperature between two successive grid points is below a certain threshold. By varying this threshold (from 0.05K to 50K), grids with varying level of refinement are obtained, whose point distribution is relatively well adapted to the problem. The finite volume mesh is then generated by taking these grid points as the positions of the cell faces.
	In this reference case, the thermal layer in the solid phase and the flame in the gas phase both have a thickness  close to $10^{-4}$m. The generated meshes are extended by adding cells with gradually increasing sizes so that the abscissa of the outer cells are ten times greater than this thickness in order to minimise the influence of the Neumann boundary conditions. It has been verified that extending the mesh further does not improve the relative error.}
	
	\modif{We show in Figure \ref{fig:convergence_CFD_vers_semi} the convergence of the CFD result towards the semi-analytical solution for the reference case.
 We see that second-order accuracy is reached, and that the relative errors reach $10^{-8}$ on $T_s$ ($10^{-7}$ on $c$ and similar results are obtained on temperature profiles) at around 4000 adapted mesh cells. For a given level of relative error, it was determined that a uniform mesh would require approximately ten times more points when using the smallest cell size of the corresponding adapted mesh. This shows that the CFD code is definitely more computationally intensive, and requires an adapted mesh to produce accurate results. Achieving a relative error lower than $10^{-8}$ on $T_s$ is difficult as this level of error is very close to the tolerance on the Newton step, i.e. the relative precision of the CFD solution obtained by the Newton solver.}

\modif{Overall, the error is sufficiently small so that we can consider that the CFD solution is converged in terms of spatial mesh and Newton iterations. The automatic mesh refinement available in the CFD tool yields similar level of errors, therefore it will be used for the rest of the comparison.
The resolution of the travelling wave problem is coherent between the two tools thus bringing out a useful cross-verification of both approaches.}
  
\subsubsection{Parametric study with variable gas phase activation energy}
\label{subsubsection:parametric_study_Ea}

We know that for the simplified chemical mechanism used, the activation energy $E_a$ of the gas phase reaction will be of paramount importance.
Indeed, if $E_a$ is low, the reaction will be very fast at lower temperatures in a narrow zone just above the surface,
which will lead to a strong heat feedback and a high regression rate.
On the opposite, if it is very large and every other parameter is unchanged, the reaction will be slower and more spread out spatially,
thus diminishing the heat feedback from the gas phase onto the solid propellant, resulting in a slower regression rate.

To highlight the effect of $E_a$, we compute with both methods the temperature profile for three values of $T_a= E_a/R$ (activation temperature),
representative of low, mid and high activation energies. The pressure remains at 5 MPa. The Lewis number is 1 for both methods. All the other parameters are not modified,
therefore neither the regression \modif{velocity},    
the surface temperature, nor the heat feedback from the flame will be the same for all three cases.

Figure \ref{fig:sweep_Ea:Tprofil} shows the spatial temperature profiles.
We see that as $T_a$ decreases, the profile becomes sharper and the flame gets closer to the surface of the propellant.
Figure \ref{fig:sweep_Ea:portrait} shows the phase portraits of these three simulations. The ordinate $\derivshort{T}{x}$ is scaled for each simulation separately,
so that the maximum is 1, otherwise the high values of $\derivshort{T}{x}$ encountered in the case $T_a=0$ would make it difficult to compare the curves.
As $T_a$ increases, the abscissa $T_s$, i.e. the propellant surface temperature, increases and so does the height of the gradient jump between the two phases.
As the activation energy is lowered, the flame becomes thinner and the heat feedback on the solid grows due to the stronger temperature gradient near the surface.
The higher surface temperature results in a greater regression rate through the pyrolysis law \eqref{eq:modifiedPyrolysisLaw}, which in turn increases the thermal
effect of the pyrolysis, i.e. the size of the gradient jump at the surface.

The fact that the gas phase portrait for $T_a=0$ is a straight line can be surprising. This is actually related to the Arrhenius law used. As stated before,
the reaction rate is of the form $\tauxreaction{} \propto [\Gi] T \exp(-T_a/T)$. Using the constant enthalpy from Proposition \ref{prop:bilan:enthalpy},
the ideal gas law and expressing the concentration $[\Gi]$ as $\rho Y/\molarmass$ in equation \eqref{eq:onde:energy},
one may easily verify that a linear function of the form $\gamma = \alpha(1-\theta)$ is a solution.


\begin{figure}
\centering
\begin{minipage}{.47\textwidth}
  \centering
  \includegraphics[width=1.1\linewidth, trim={0cm 0 0 0},clip]{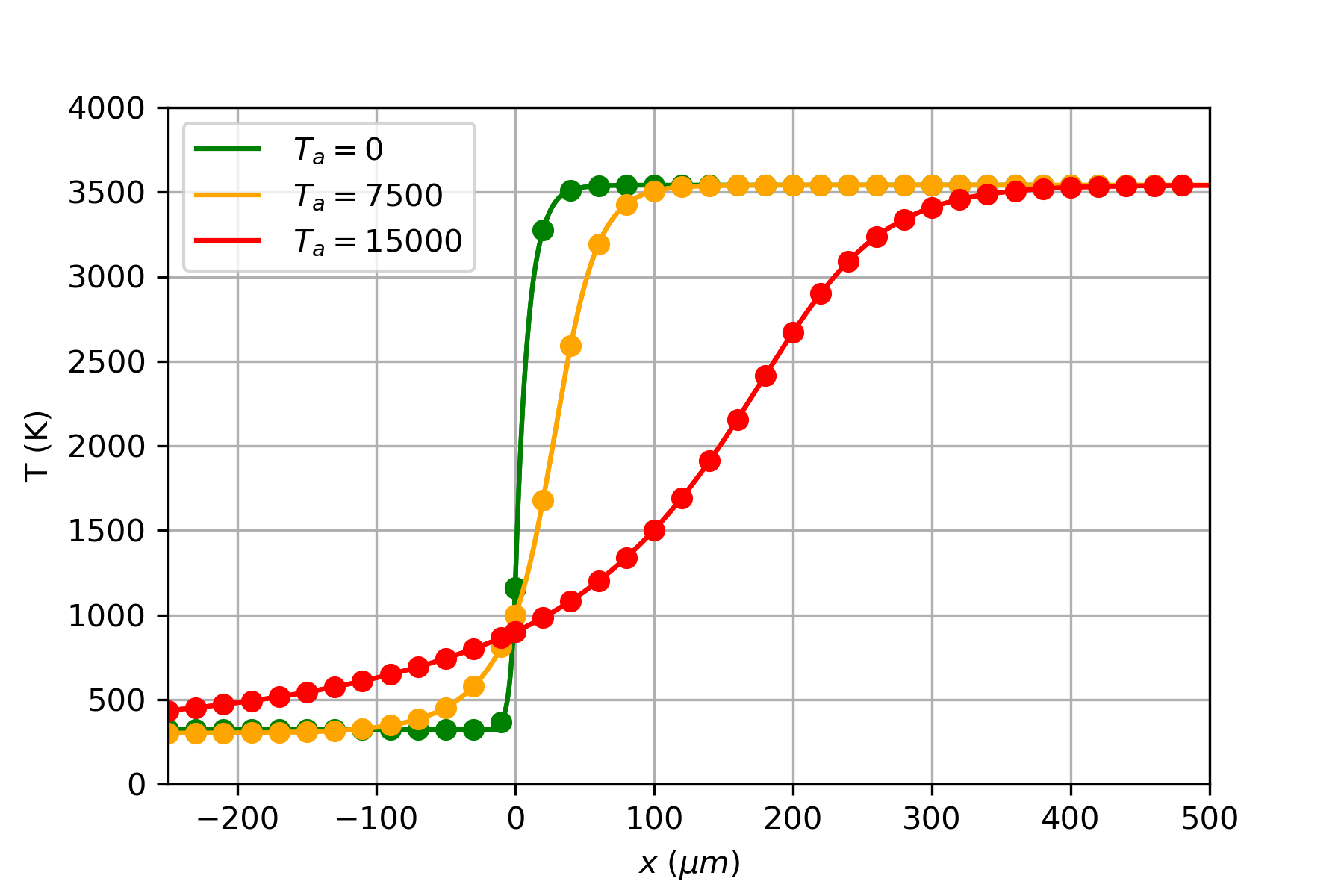}
  \captionof{figure}{Temperature profiles, CFD results (full lines) compared to semi-analytical results (circles) }
  \label{fig:sweep_Ea:Tprofil}
\end{minipage}%
\hfill
\begin{minipage}{.47\textwidth}
  \centering
  \includegraphics[width=1.1\linewidth, trim={0cm 0cm 0 0cm},clip]{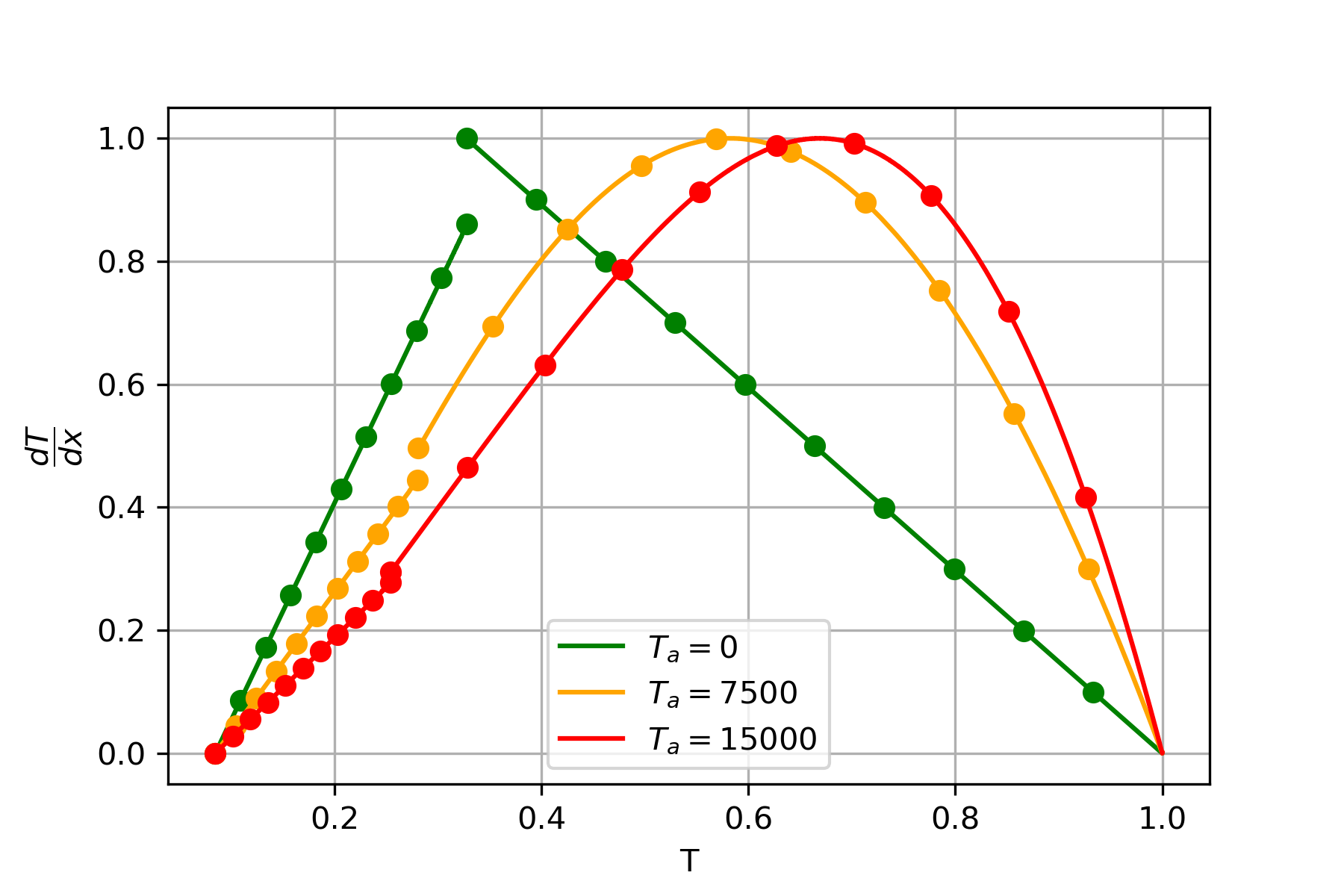}
  \captionof{figure}{Normalised phase portraits, CFD results (full lines) compared to semi-analytical results (circles) }
  \label{fig:sweep_Ea:portrait}
\end{minipage}
\end{figure}

A more thorough parametric study has been performed to obtain Figure \ref{fig:sweep_Ea:regression}. The agreement of both methods for the prediction of the
regression speed
 is very good across the whole range of activation temperatures, \modif{with a relative error of approximately $10^{-7}$ on $c$.}
An important remark is that the CFD solution often fails \modif{to converge when the initial mesh is not suited, and when the initial solution is not sufficiently good}.
For example, the case $T_a=0$ involves very strong temperature gradients, which required adding many more mesh points close the surface for the initial solution.
On the opposite, the case $T_a = 15000$ K gives a very smooth and slowly evolving temperature profile, but this translates to a very spread out flame, requiring
additional mesh points far from the surface so that the combustion process is fully represented.
Rather than remedying these problems manually \modif{by using a single extended mesh and performing transient iterations to facilitate convergence, we use the semi-analytical method to generate the initial solution, and define an initial mesh as explained in Section \ref{subsubsection:CFD_convergence} with a sufficient extension so that the gas phase reaction is completed within the computational domain. With this approach, the CFD code converges very quickly and can further refine the mesh if needed.
This highlights some of the main advantages of the semi-analytical method, which are that the solution always converges, and that no manual mesh adjustments are needed.}

Figure \ref{fig:sweep_Ea:regression} also shows the results of the analytical models WSB and DBW. Their pre-exponential factors $A_p$ and $A$ were adjusted so that 
both models predict the same regression \modif{velocity}    
 at $T_a=7216$K as the CFD model. The WSB model assumes $T_a=0$, therefore the regression \modif{velocity}    
 does not vary with $T_a$.
We see that the tendencies are reasonable, even if not in perfect  agreement  (we use a log scale), between the semi-analytical model and the DBW model for
high activation energies. However the DBW model, which assumes high gas phase activation energy, falls apart when $T_a$ is decreased.
Overall, the semi-analytical model is a more generic model that produces quantitatively good results, without any assumption on $T_a$.

\begin{figure}
\centering
\begin{minipage}{.47\textwidth}
  \centering
  \includegraphics[width=1.1\linewidth, trim={0cm 0 0 0},clip]{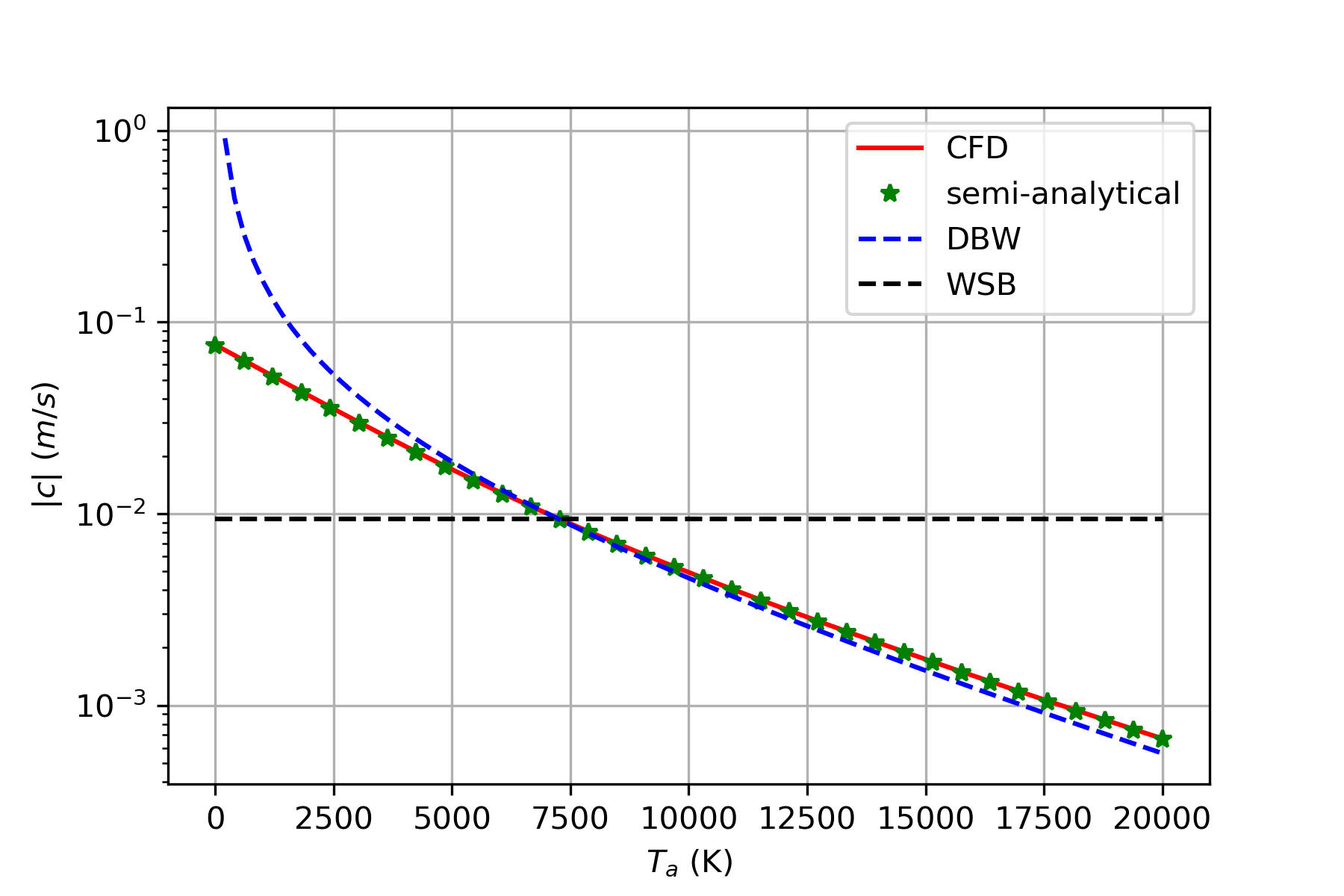}
  \captionof{figure}{Regression speed as a function of activation temperature, CFD results compared to semi-analytical and analytical results}
  \label{fig:sweep_Ea:regression}
\end{minipage}
\hfill
\begin{minipage}{.47\textwidth}
  \centering
  \includegraphics[width=1.11\linewidth, trim={0cm 0 0 0.95cm},clip]{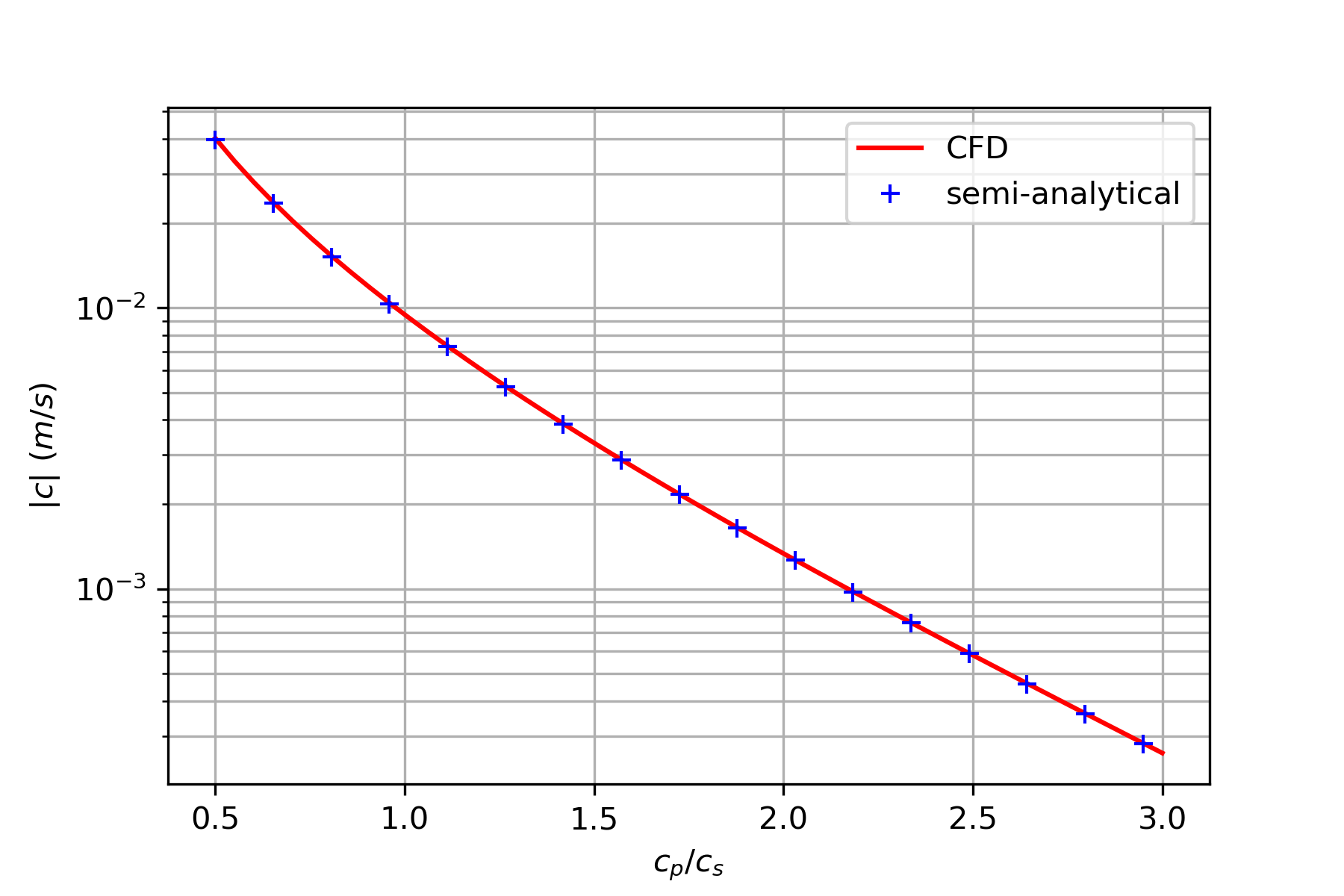}
 \captionof{figure}{Evolution of the regression speed with $c_p/c_s$ obtained with the semi-analytical and CFD methods}
  \label{fig:sweep_cpcs:c_vs_cpcs}
\end{minipage}%
\end{figure}

\subsubsection{Parametric study on the ratio $c_p/c_s$}
\label{subsubsection:parametric_study_cpcs}

We now wish to extend the numerical method beyond its theoretical ground, by relaxing  Assumption H\ref{assum:cpcs}.
Simulations are performed with the CFD code and our semi-analytical tool, by varying the gas specific heat $c_p$ at constant
$c_s=1.2 \times 10^3 \textnormal{J.kg}^{-1} \textnormal{K}^{-1}$.
The species diffusion coefficient $D_g$ is taken as a linear function of $T$, such that the Lewis number remains equal to unity. Therefore it varies with $c_p$.
In our semi-analytical tool, we account for $c_p \neq c_s$ by using the equation \eqref{eq:definition:Scpcs} instead of \eqref{eq:definition:S} for $S$, 
and $T_f$ and $\Qpyro$ are computed as in remarks \ref{remark:energyBalanceCsCp} and \ref{assum:cpcs}.
The ODE \eqref{eq:portrait:gas} changes slightly as the ratio $c_s/c_p$ appears as an additional factor for the term in $\eta c$, \modif{which also affects the slope \eqref{eq:slope_critic} of the solution near the critical point.}
The results are shown in Figure \ref{fig:sweep_cpcs:c_vs_cpcs} for a wide range of ratios $c_p/c_s$ ($0.5$ to $3$), \modif{which encompasses all physically relevant solid propellant configurations.}
We see that the CFD code and the semi-analytical model are again in very close agreement.
\modif{The relative error between both tools is around $10^ {-8}$ on the surface temperature ($10^{-7}$ on $c$).}

\modif{As we do not have a theoretical proof of existence and uniqueness when $c_p \neq c_s$, we have performed a more extensive study to observe the behaviour of $\xi$ when we vary the ratio $c_p/c_s$, even if it is outside of the physically relevant interval.
The curve of $\xi(c)$ is plotted for various ratios $c_p/c_s$ in Figure \ref{fig:various_cpcs}. Each curve is normalized by $\xi(0)$, the limit of $\xi$ when $c$ tends to 0. We see that $\xi$ remains monotonous and only has one zero-crossing.
Figure \ref{fig:cpcs_limites} shows how the solution wave speed $c_{sol}$ is located between the bounds $\cmax$ and $\cmin$ as the ratio $c_p/c_s$ changes. We observe that the solution remains within these bounds, and tends to $\cmax$ for high values of the ratio $c_p/c_s$. When this ratio is low, both the solution and $\cmin$ tend to $\cmax$.
Overall, this numerical investigation shows that the semi-analytical model can still be reliable beyond the simplified level of modelling adopted for the theoretical analysis.
}
       \begin{figure}[hpbt!]
        \begin{subfigure}[t]{0.5\textwidth}
	        \centering
	        \includegraphics[width=\textwidth]{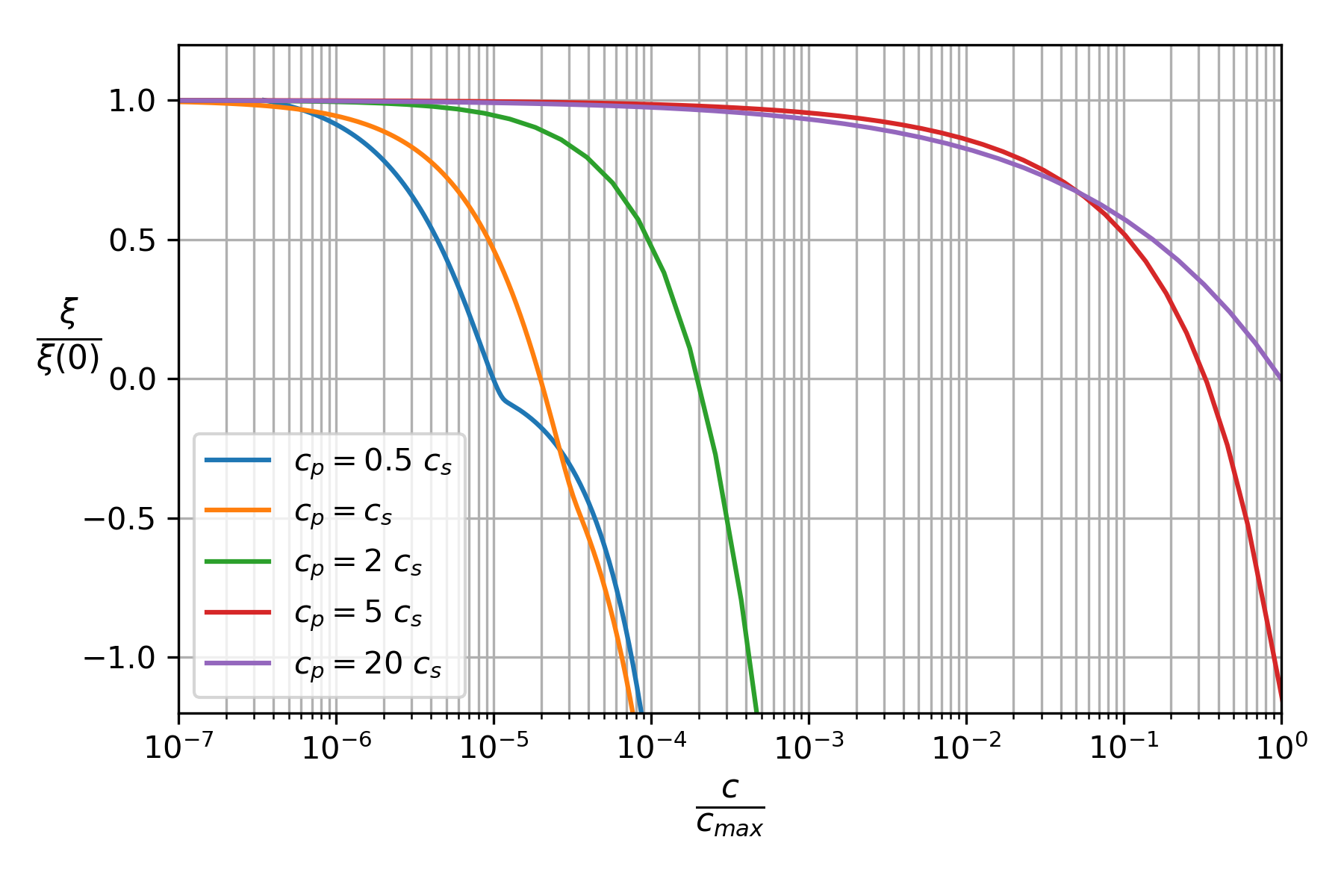}
	        \caption{\centering
	        Evolution of $\xi$ with $c$ for different ratios of $c_p/c_s$}
	        \label{fig:various_cpcs}
        \end{subfigure}
        ~
        \begin{subfigure}[t]{0.5\textwidth}
	        \centering
        	\includegraphics[width=\textwidth]{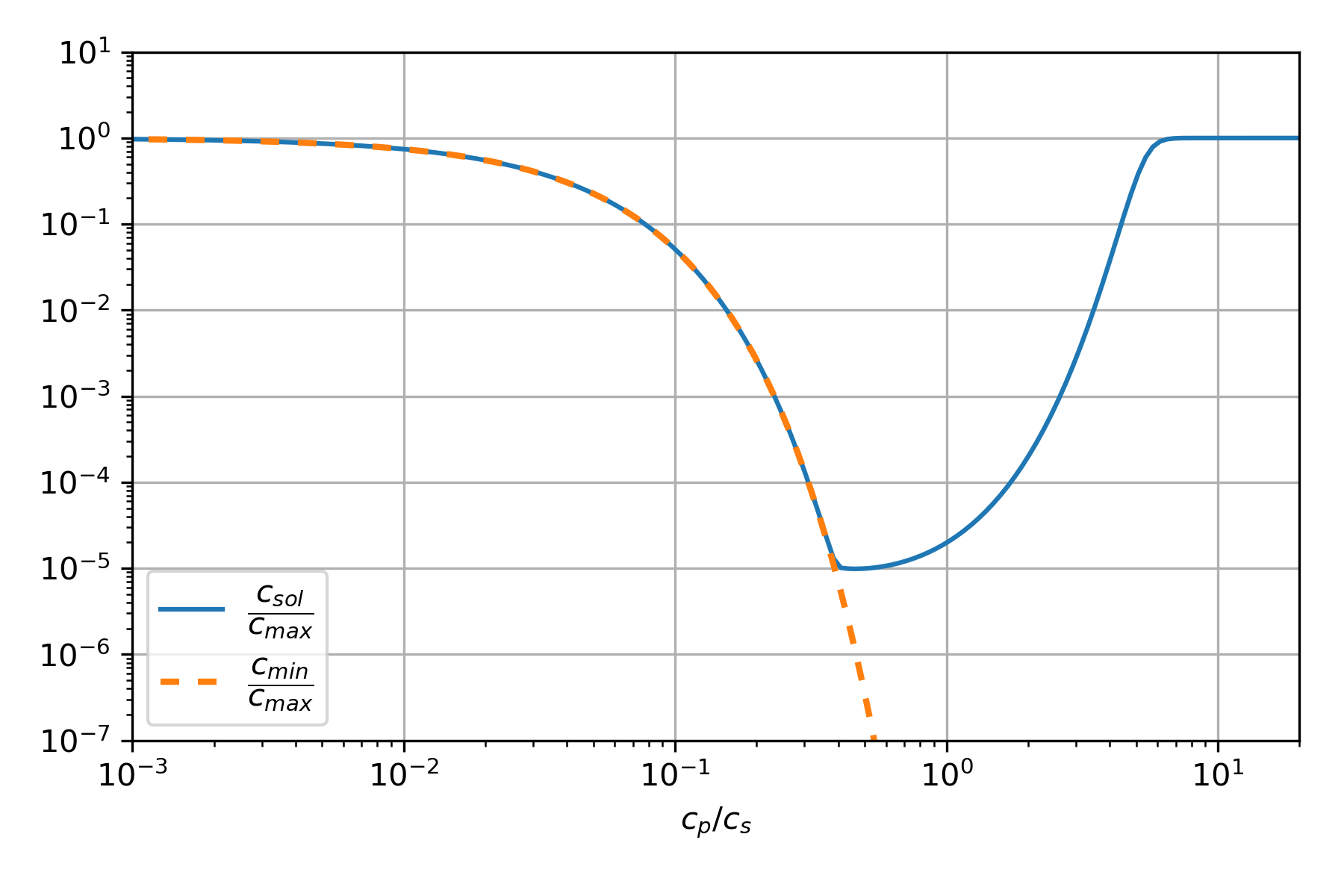}
	        \caption{
	        \centering
	       Evolution of the solution regression speed $c_{sol}$ (blue dashes) and lower bound $\cmin$ (orange line) compared to $\cmax$}
	        \label{fig:cpcs_limites}
        \end{subfigure}
        \caption{Effects of the ratio $c_p/c_s$ on $\xi$ and on the solution regression speed}
      \end{figure}

\subsubsection{Parametric study on the Lewis number}
\label{subsubsection:parametric_study_Lewis}
\modif{The unitary Lewis number assumption H\ref{assum:lewis} allows to simplify the system by only having to consider the temperature and its gradient as variables. This can be relaxed numerically, however it would make the shooting method more complex, requiring the addition of the mass fraction $Y$ and its gradient as variables in the dynamical system equations. It may also affect the existence and uniqueness of the solution. Zeldovich has reported that uniqueness can indeed be lost for laminar flames \cite{livreZeldovich} when $\mathrm{Le}>1$.}
To show the limits of the semi-analytical model, we conduct a study on the effect of a constant Lewis number, but with a value different than 1.
To do so, just as before, the CFD code uses species diffusion coefficients that are linear with $T$, such that $\mathrm{Le}$ is constant in the gas phase.

If $\mathrm{Le}$ is high, heat diffuses faster than species, therefore we expect a stronger thermal feedback from the gas phase, 
resulting in a faster regression speed.
When the Lewis number is decreased below 1, we expect the opposite effect. 
This is confirmed in Figure \ref{fig:sweep_Lewis:Tprofil} which shows the temperature profiles for three different values of $\mathrm{Le}$.
Figure \ref{fig:sweep_Lewis:error} shows the relative error of the semi-analytical model for the estimation of $c$,
compared to the CFD result, for Lewis numbers within the realistic range from $0.5$ to $3$.
As expected, the minimum error is reached \modif{at} $\mathrm{Le}=1$. 
For $\mathrm{Le}>1$, the semi-analytical model underestimates $c$ as it underestimates the temperature gradients near the surface. For $\mathrm{Le}<1$, $c$ is overestimated.
Still, the relative error on $c$ lies within $20\%$. The relative error on $T_s$ is much smaller ($\approx 1\%$).
The exponential term in the pyrolysis law with the high pyrolysis activation temperature
$T_{ap}$ is the root of this difference, as a small relative error in the evaluation of $T_s$ translates into a greater one for $c$.
Overall, this parametric study shows that the unitary Lewis number assumption still allows for a quantitatively reasonable solution.

\begin{figure}
\centering
\begin{minipage}{.47\textwidth}
  \centering
  \includegraphics[width=1.1\linewidth, trim={0cm 0 0 0},clip]{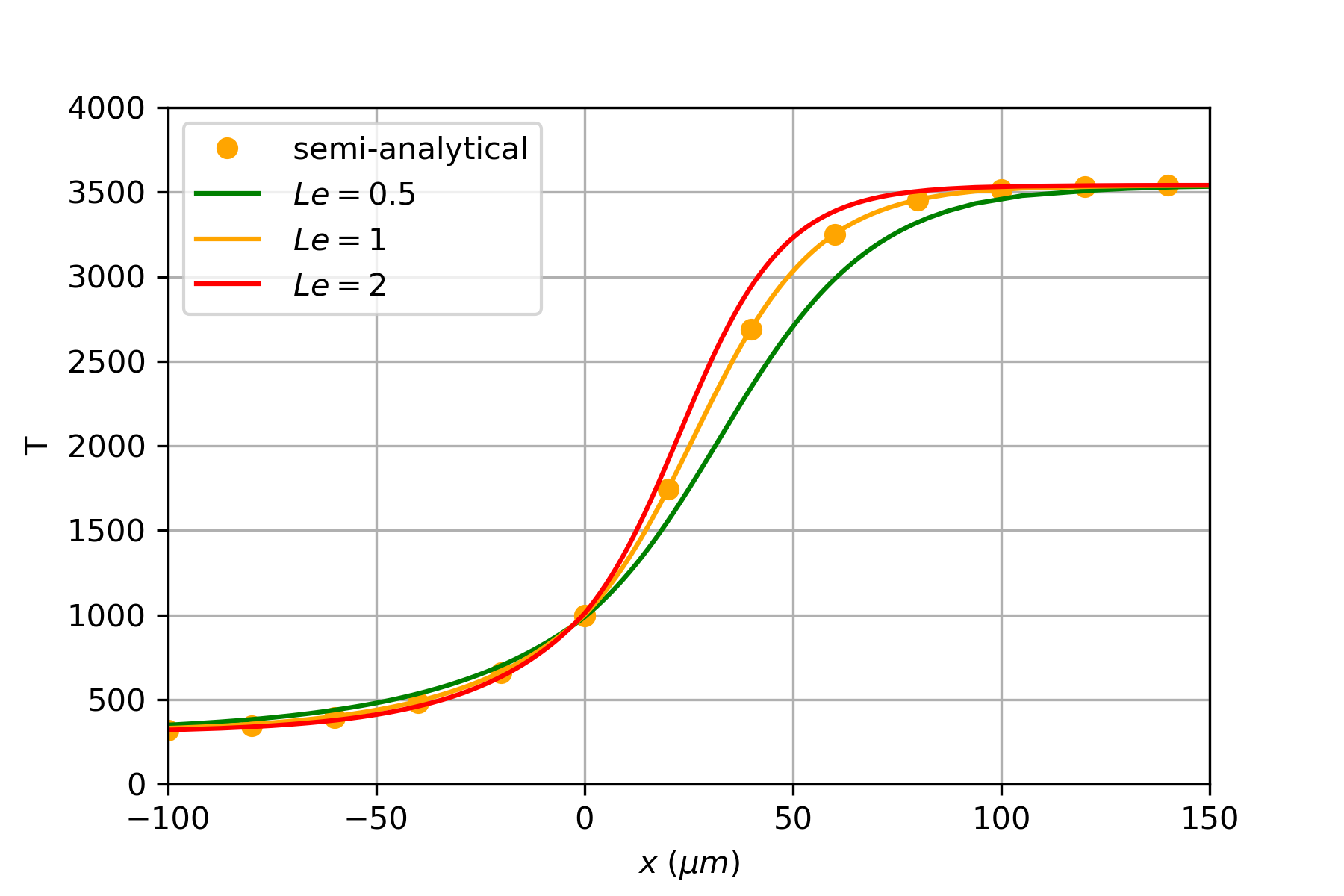}
  \captionof{figure}{Temperature profiles, CFD results (full lines) compared to semi-analytical results (dots) for different values of the Lewis number}
  \label{fig:sweep_Lewis:Tprofil}
\end{minipage}%
\hfill
\begin{minipage}{.47\textwidth}
  \centering
  \includegraphics[width=1.1\linewidth, trim={0cm 0cm 0 0cm},clip]{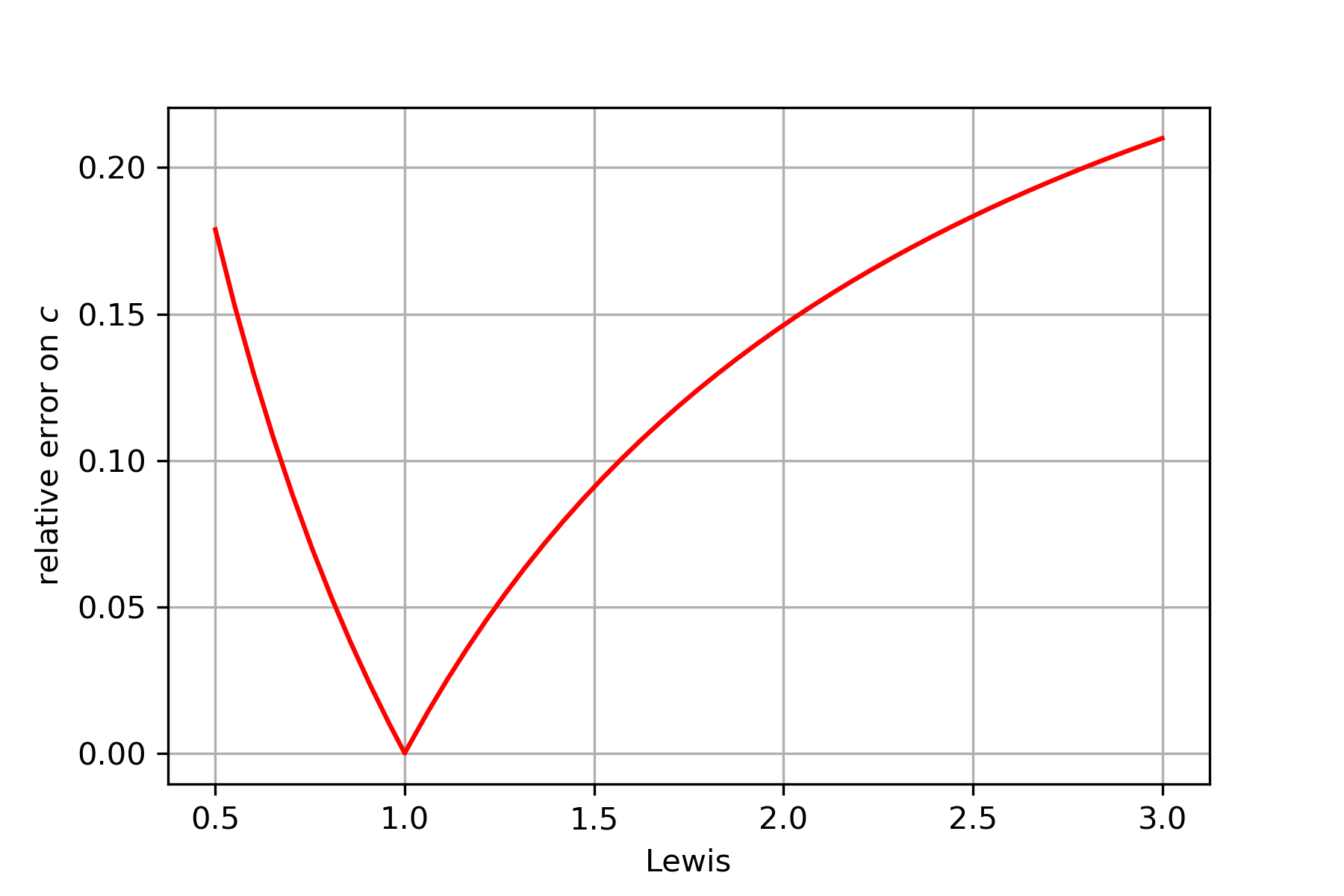}
  \captionof{figure}{Relative error on the prediction of $c$ by the semi-analytical method compared to the CFD method, as a function of the Lewis number}
  \label{fig:sweep_Lewis:error}
\end{minipage}
\end{figure}

\section{Conclusion}
We have presented a new method for  the determination of \modif{travelling wave solutions for} a simplified combustion model of a homogeneous solid propellant.
The main assumptions are that the gas phase only contains one reactant and one product,  the reactant being transformed into the product
by a single irreversible reaction, and that the Lewis number is 1.
 \modif{Considering solutions of this nonlinear eigenvalue problem in the form of a travelling wave profile as well as a regression velocity $c$}, 
we have derived a reduced system
which can be used in a numerical shooting method to determine the correct regression \modif{velocity}.
We have proven that the travelling wave solution profile and \modif{velocity} exist and are unique under conditions which are not restrictive
in view of the physical properties encountered in real solid propellants.

A numerical comparison has been conducted with a CFD code developed at \mbox{ONERA}, and the agreement is very good for a broad range of parameter values,
at least for a unitary Lewis. \modif{The shooting method is simpler to implement, more efficient and reliable than the CFD code}. We have shown that the relative error on $c$ grows as the Lewis number changes, but the \modif{semi-analytical} solution remains quantitatively correct for realistic values of this parameter. A comparison has also been made with some of the main analytical models, and we have shown that our semi-analytical model produces better results overall.
Practically, the semi-analytical method is free of any space discretisation error. The remaining sources of error are well controlled.
This method can thus be a useful verification tool for CFD codes with simplified test cases.
Besides, the method always converges, hence it can be used to generate initial solutions for more detailed methods that would otherwise struggle to converge.

This method can be employed to determine the various coefficients needed to compute the linear response function to pressure fluctuations,
by performing multiple simulations with slight variations of one parameter.

The proof of existence and uniqueness may be extended to include the effect of a constant external heat flux absorbed both at the surface and inside the solid. Care must be taken,
as the burnt gas temperature $T_f$ will depend on the mass flow rate.
Also, as in \cite{1963JohnsonNachbarUnicity}, a reverse reaction may be included in the gas phase, allowing for a non-trivial equilibrium far away from the surface.
These effects were not included in the present paper for readability reasons.

A few evolutions can be envisioned for the numerical shooting method.
More advanced pyrolysis laws can be used, as in \cite{brewster2004}.
Radiation penetration and absorption, as well as heat loss by \modif{thermal radiation} and potential heat loss to the surroundings may be easily integrated into the numerical tool.
This will require a numerical resolution of the solid phase temperature profile, as already done for the gas phase, instead of the simple analytical solution that we have been able to use in this paper.
\modif{The shooting method may also be adapted to take into account a temperature dependent Lewis number.}
\modif{Although not detailed in this paper, we have conducted a numerical experiment on the inclusion of non-linear surface phenomena. For example, following the work of Johnson and Nachbar \cite{deflagrationLimitJohnson}, a radiative heat loss on the propellant surface can be included. This heat loss is accounted for in $S$ and also brings a new dependence of $T_f$ on $c$. It is observed that $\xi$ is not a monotonous function any more: depending on the surface emissivity, there can be one solution, two, or none. When there are two solutions, only one is thought to be stable, but the shooting method is still able to find the potentially unstable one. This shows some potential for the semi-analytical tool.}
\modif{When the shooting method is extended to account for such phenomena, the existence and uniqueness of the solution might not be guaranteed any more, and the bounds on the regression velocity may need to be redefined. Finally, the dichotomy process may need to be performed on multiple separate intervals for the regression velocity to allow the determination of all solutions.}

We believe that \modif{the approach presented in this paper} also sets the proper framework for the stability analysis of the stationary wave profile;
this is out of the scope of the present paper but is the subject of future work. 

\section*{Acknowledgements}
The present research was conducted thanks to a Ph.D grant co-funded by DGA, Ministry of Defence (E. Faucher, Technical Advisor), and ONERA.
The author would like to thank V. Giovangigli for several helpful discussions,
\modif{and the reviewers for very useful remarks, which have helped improving the paper}.

\bibliographystyle{unsrt}
\bibliography{references}

\begin{thebibliography}{10}

\bibitem{deluca1992}
L.~De Luca.
\newblock {\em Theory of Nonsteady Burning and Combustion Stability of Solid
  Propellants by Flame Models}, chapter~14, pages 519--600.
\newblock American Institute of Aeronautics and Astronautics, 1992.

\bibitem{novozhilov1992}
B.V. Novozhilov.
\newblock {\em Theory of Nonsteady Burning and Combustion Stability of Solid
  Propellants by the Zeldovich-Novozhilov Method}, chapter~15, pages 601--641.
\newblock American Institute of Aeronautics and Astronautics, 1992.

\bibitem{lengelle}
G.~Lengelle, J.~Duterque, and J.F. Trubert.
\newblock Physico-chemical mechanisms of solid propellant combustion.
\newblock {\em Solid Propellant Chemistry, Combustion, and Motor Interior
  Ballistics}, 185:287--334, 01 2000.

\bibitem{barrere}
F.A. Williams, M.~Barrere, and N.C. Huang.
\newblock {\em Fundamental aspects of solid propellant rockets}.
\newblock the Advisory Group for Aerospace Research and Development of
  N.A.T.O., Technivision Services; Distributed by Technical P. Slough; London,
  1969.

\bibitem{williamsDBW}
F.A. Williams.
\newblock Quasi-steady gas-phase flame theory in unsteady burning of a
  homogeneous solid propellant.
\newblock {\em AIAA Journal}, 11(9):1328--1330, 1973.

\bibitem{BDPmodel}
M.W. Beckstead, R.L. Derr, and C.F. Price.
\newblock A model of composite solid-propellant combustion based on multiple
  flames.
\newblock {\em AIAA Journal}, 8(12):2200--2207, 1970.

\bibitem{WSBmodel}
M.J. Ward, S.F. Son, and M.Q. Brewster.
\newblock Steady deflagration of {HMX} with simple kinetics: A gas phase chain
  reaction model.
\newblock {\em Combustion and Flame}, 114(3):556 -- 568, 1998.

\bibitem{Brewster_simplified_combustion}
M.Q. Brewster.
\newblock Combustion mechanisms and simplified-kinetics modeling of homogeneous
  energetic solids.
\newblock {\em Energetic Materials: Part 2. Detonation, Combustion},
  13(8):225--294, 2003.

\bibitem{schatzman2002}
{M.} Schatzman.
\newblock {\em Numerical Analysis: A Mathematical Introduction}, volume 140.
\newblock Oxford University Press, 2002.

\bibitem{henderson1997}
J.~Henderson and H.~Wang.
\newblock Positive solutions for nonlinear eigenvalue problems.
\newblock {\em J. Math. Anal. Appl.}, 208(1):252--259, 1997.

\bibitem{massa_2D}
L.~Massa, T.L. Jackson, and J.~Buckmaster.
\newblock New kinetics for a model of heterogeneous propellant combustion.
\newblock {\em Journal of Propulsion and Power}, 21(5):914--924, 2005.

\bibitem{meynet2005simulation}
N.~Meynet.
\newblock {\em Simulation num{\'e}rique de la combustion d'un propergol
  solide}.
\newblock PhD thesis, Universit{\'e} Pierre et Marie Curie (Paris), 2005.

\bibitem{meynet2006}
V.~Giovangigli, N.~Meynet, and M.D. Smooke.
\newblock Application of continuation techniques to ammonium perchlorate plane
  flames.
\newblock {\em Combustion Theory and Modelling}, 10(5):771--798, 2006.

\bibitem{livreZeldovich}
Ya.B. Zeldovich, G.I. Barenblatt, V.B. Librovich, and G.M. Makhviladze.
\newblock {\em The Mathematical Theory of Combustion and Explosion}.
\newblock Plenum Publishing, 1985.

\bibitem{Verri}
M.~Verri.
\newblock Asymptotic stability of traveling waves in solid-propellant
  combustion under thermal radiation.
\newblock {\em Mathematical Models and Methods in Applied Sciences},
  09:1279--1305, 12 1999.

\bibitem{1963JohnsonNachbarUnicity}
W.E. {Johnson} and W.~{Nachbar}.
\newblock Laminar flame theory and the steady, linear burning of a
  monopropellant.
\newblock {\em Archive for Rational Mechanics and Analysis}, 12:58--92, January
  1963.

\bibitem{degreTopologique_Berestycki}
H.~Berestycki, B.~Nicolaenko, and B.~Scheurer.
\newblock Traveling wave solutions to combustion models and their singular
  limits.
\newblock {\em SIAM Journal on Mathematical Analysis}, 16(6):1207--1242, 1985.

\bibitem{giovangigli1999}
V.~Giovangigli.
\newblock Plane laminar flames with multicomponent transport and complex
  chemistry.
\newblock {\em Mathematical Models and Methods in Applied Sciences},
  9(3):337--378, 1999.

\bibitem{TanakaBecksteadAP}
M.~Tanaka and M.W. Beckstead.
\newblock A three-phase combustion model of ammonium perchlorate.
\newblock {\em 32nd Joint Propulsion Conference and Exhibit}, page 2888, 1996.

\bibitem{DavidsonBecksteadHMX}
J.E. Davidson and M.W. Beckstead.
\newblock A three-phase model of {HMX} combustion.
\newblock {\em Symposium (International) on Combustion}, 26(2):1989 -- 1996,
  1996.

\bibitem{DavidsonBecksteadRDX}
J.E. Davidson and M.W. Beckstead.
\newblock Improvements in rdx combustion modeling.
\newblock {\em 32nd JANNAF Combustion Meeting}, 1:41--56, 01 1995.

\bibitem{coldBoundaryKarman1953}
T.~von Kármán and G.~Barbany.
\newblock The thermal theory of constant pressure deflagration.
\newblock In {\em Anniversary Volume on Applied Mechanics dedicated to C. B.
  Biezeno by some of his Friends and Former Students on the Occassion on his
  Syxty-Five Birthday, March 2, 1953}, pages 58--69. H. Stam, Haarlem, 1953.

\bibitem{coldBoundaryMarble1956}
F.E. Marble.
\newblock Flame theory and combustion technology.
\newblock {\em Journal of the Aeronautical Sciences}, 23(5):462--468, 1956.

\bibitem{coldBoundary1991}
H.~Berestycki, B.~Larrouturou, and J.M. Roquejoffre.
\newblock Mathematical investigation of the cold boundary difficulty in flame
  propagation theory.
\newblock In P.C. Fife, Amable Li{\~{n}}{\'a}n, and F.~Williams, editors, {\em
  Dynamical Issues in Combustion Theory}, pages 37--61, New York, NY, 1991.
  Springer New York.

\bibitem{deflagrationLimitJohnson}
W.E. Johnson and W.~Nachbar.
\newblock Deflagration limits in the steady linear burning of a monopropellant
  with application to ammonium perchlorate.
\newblock {\em Symposium (International) on Combustion}, 8(1):678 -- 689, 1961.
\newblock Eighth Symposium (International) on Combustion.

\bibitem{theseShihab}
S.~Rahman.
\newblock {\em Mod\'elisation et simulation num\'erique de flammes planes
  instationnaires de perchlorate d'ammonium}.
\newblock PhD thesis, Universit\'e Pierre et Marie Curie, 2012.

\bibitem{jackson_2D}
T.L. Jackson and J.~Buckmaster.
\newblock Heterogeneous propellant combustion.
\newblock {\em AIAA Journal}, 40(6):1122--1130, 2002.

\bibitem{Miller_idealized}
M.S. Miller.
\newblock In search of an idealized model of homogeneous solid propellant
  combustion.
\newblock {\em Combustion and Flame}, 46:51 -- 73, 1982.

\bibitem{IBIRICU1975185}
F.A.~Williams M.M.~Ibiricu.
\newblock Influence of externally applied thermal radiation on the burning
  rates of homogeneous solid propellants.
\newblock {\em Combustion and Flame}, 24:185 -- 198, 1975.

\bibitem{berestycki1994}
H.~Berestycki, L.~Nirenberg, and S.R.S. Varadhan.
\newblock The principal eigenvalue and maximum principle for second-order
  elliptic operators in general domains.
\newblock {\em Communications on Pure and Applied Mathematics}, 47(1):47--92,
  1994.

\bibitem{Volpert3}
{A.I.} Volpert, {V.A.} Volpert, and {V.A.} Volpert.
\newblock {\em Traveling Wave Solutions of Parabolic Systems: Translations of
  Mathematical Monographs}, volume 140.
\newblock American Mathematical Society, 1994.

\bibitem{hairer_radau5}
E.~Hairer and G.~Wanner.
\newblock {\em Solving Ordinary Differential Equations {I}{I}. Stiff and
  Differential-Algebraic Problems}, volume~14 of {\em Springer Series in
  Comput. Math.}
\newblock Springer-Verlag Berlin Heidelberg, 2nd edition, 1996.

\bibitem{Brent1973}
R.P. Brent.
\newblock {\em {Algorithms for Minimization without Derivatives}}.
\newblock Prentice-Hall, Englewood Cliffs, New Jersey, 1st edition, 1973.

\bibitem{Smooke_Gio}
M.D. Smooke and V.~Giovangigli.
\newblock Numerical modeling of axisymmetric laminar diffusion flames.
\newblock {\em IMPACT of Computing in Science and Engineering}, 4(1):46 -- 79,
  1992.

\bibitem{dmitri_modele1D}
M.~Muller, D.~Davidenko, and V.~Giovangigli.
\newblock Computational study of aluminum droplet combustion in different
  atmospheres.
\newblock In {\em 7th European Conference for Aeronautics and Space Sciences
  (EUCASS)}, pages 1--17, 07 2017.

\bibitem{ref_CHEMKIN2}
R.J. Kee, F.M. Rupley, and J.A. Miller.
\newblock Chemkin-{II}: A fortran chemical kinetics package for the analysis of
  gas-phase chemical kinetics.
\newblock 9 1989.

\bibitem{EGlib}
V.~Giovangigli and A.~Ern.
\newblock Eglib: A multicomponent transport software for fast and accurate
  evaluation algorithms.
\newblock {\em Manual, Ecole polytechnique}, 2004.
\newblock \url{http://www.cmap.polytechnique.fr/www.eglib/}.

\bibitem{brewster2004}
T.~Jackson, L.~Massa, and M.Q. Brewster.
\newblock Unsteady combustion modelling of energetic solids, revisited.
\newblock {\em Combustion Theory and Modelling}, 8:513--532, 09 2004.

\end{thebibliography}

\appendix
\section{Analysis with constant surface temperature}
\label{remark:casParticulierTsConstante}
Johnson and Nachbar \cite{1963JohnsonNachbarUnicity} proved the uniqueness of $c$ for any fixed value of $T_s$.
Using the same approach as in Propositions \ref{prop:existenceSolution}, \ref{prop:uniquenessQp_negatif}, \ref{prop:uniquenessQp_positif}, one may
prove this result by following the steps listed below.
 \begin{itemize}
  \item We replace the pyrolysis law \eqref{eq:modifiedPyrolysisLaw} by the trivial relation $T_s = T_{s,0}$. As a consequence, and 
 even without Assumption H\ref{assum:cpcs} ($c_p=c_s$), $\Qpyro$ is a constant, whence, from equation \eqref{eq:definition:Scpcs}, $S$ is a linear function of $c$,
 as it is the case in our previous study.
 
 \item  Before proving the existence of a solution, it is actually useful to show that the gas heat feedback $\gamma^+$ is a monotonous function of $c$.
 Similarly to what is done in the proof of Proposition \ref{prop:uniquenessQp_negatif}, we can write $\derivshort{\gamma^+}{c} = A > 0$.
 The term $B$ which appeared in the proof of Proposition \ref{prop:uniquenessQp_negatif} is strictly $0$ as $\derivshort{\theta_s}{c}=0$.
 Consequently, $\gamma^+$ is a monotonous function of $c$: as $c$ increases ($\massflux$ decreases), the gas heat feedback increases.
 The monotonicity of $\gamma^+$ will be helpful when discussing the existence of a solution.
 
 \item  Let us express the \nomXi{}:
 $$\xi(c) = \Delta\gamma(c) - S(c) = \gamma^+(c) - \eta\gamma^-(c) - S(c)$$
 The analytical solution of the solid phase temperature profile gives $\gamma^-(c) = -c \theta_s$.
 Using the definition of $S$ in equation \ref{eq:definition:Scpcs}, we can reformulate $\xi$ as:
 \begin{equation}
 	\label{annexJN:xi}
	\xi(c) = \gamma^+ + \eta c \underbrace{\left( \theta_s - \dfrac{\Qpyro}{c_s(T_f-T_0)} \right)}_{K} 
 \end{equation}

  Note that $K>0$, as we still retain the property $T_s \geq \Tslim$ (see proof of Proposition \ref{prop:uniquenessQp_positif}).
  
 \item To establish the existence of a solution wave \modif{velocity} $c$ such that $\xi(c)=0$,
 we rely on two limit cases, $c=0$ and $c=\cmax$, as we did in the proof of Proposition \ref{prop:existenceSolution}:
 
 \begin{itemize}
    \item The first limit case is $c=0$, and we easily obtain $\xi(0) = \Delta\gamma(0) = \left( 2 \int_{\theta_s}^{1} \Psi(y) dy \right)^{1/2} > 0$
    as in the proof Proposition \ref{prop:existenceSolution}.

    \item The second limit case is $c=\cmax$. In our previous study, $\cmax$ was a value obtained via the pyrolysis law for $T_s=T_f$.
    In the study of Johnson and Nachbar, $T_s$ does not depend on $c$ any more, therefore we use $\cmax=-\infty$.
    The monotonicity of $\gamma^+$ we have just established before leads to $\gamma^+(-\infty) < \gamma^+(0)$.
    Moreover, the monotonicity of the temperature profile also yields $\gamma^+(-\infty) \geq 0$.
    As $K > 0$, we deduce from \eqref{annexJN:xi} that $\xi(-\infty) = -\infty$.
\end{itemize}
 Overall, we have $\xi(0) > 0$ and $\xi(-\infty)=-\infty$, therefore we conclude that a solution wave \modif{velocity} $c$ exists, such that $\xi(c)=0$.
 
 \item To prove the uniqueness of the solution, we derive $\xi$ with respect to $c$:
  $$\derivshort{\xi}{c} = \derivshort{\gamma^+}{c} + \eta K $$
  We have shown that all the terms are positive, hence $\derivshort{\xi}{c} > 0$. We conclude that the solution is unique.
 \end{itemize}

\end{document}